\documentclass[12pt]{amsart}
\usepackage{amsthm,amsopn,amsfonts,amssymb,pb-diagram}
\usepackage{hyperref}
\usepackage{fancyhdr}
\usepackage[dvips]{graphics}
\usepackage[usenames,dvipsnames,svgnames,table]{xcolor}
\usepackage{pict2e} % allows you to draw short diagonal lines
\usepackage{tikz}

\usepackage{euler, amsfonts, amssymb, latexsym, epsfig, epic, enumerate}
\usepackage{epstopdf}

\usepackage[OT2,OT1]{fontenc}
\newcommand\cyr{%
  \renewcommand\rmdefault{wncyr}%
  \renewcommand\sfdefault{wncyss}%
  \renewcommand\encodingdefault{OT2}%
  \normalfont
  \selectfont}
\DeclareTextFontCommand{\textcyr}{\cyr}

\font\co=lcircle10
\def\boxcross{\ \smash{\lower6.5pt\hbox{\rlap{\hskip4.5pt\vrule height13.5pt}}
                \raise0pt\hbox{\rlap{\hskip-2pt \vrule height.4pt depth0pt
                        width13.5pt}}}\hskip12.7pt}

\def\boxelbow{\ \hskip.1pt\smash{%
               \hbox{\co \hskip 5.5pt\rlap{\mathsurround=0pt\rlap{\mathsurround=0pt\char'006}\lower0.4pt\rlap{\char'004}}
                \lower6.5pt\rlap{\hskip-0.2pt\vrule height3pt}
                \raise3.5pt\rlap{\hskip-0.2pt\vrule height3.2pt}}
                \hbox{%
                  \rlap{\hskip-6.4pt \vrule height.4pt depth0pt
width2.5pt}%
                  \rlap{\hskip4.05pt \vrule height.4pt depth0pt
width3.1pt}}}
                \hskip8.7pt}

\newtheorem{Theorem}{Theorem}[section]
\newtheorem{thm}[Theorem]{Theorem}
\newtheorem*{Theorem*}{Theorem}
\newtheorem{Lemma}[Theorem]{Lemma}

\newtheorem{Proposition}[Theorem]{Proposition}

\theoremstyle{definition}

\theoremstyle{remark}
\newtheorem{example}[Theorem]{Example}

\newcommand{\pt}{\mathrm{pt}}

\newcommand{\Z}{\mathbb Z}

\newcommand{\weight}{{\rm wt}}
\newcommand{\fusing}{{\rm fusing}}

\newcommand{\rowspan}{{\rm rowspan}}

\newcommand{\ZZ}{\mathbb{Z}}

\newcommand{\fixit}[1]{\texttt{*** #1 ***}}

\newcommand\defn[1]{{\bf #1}}

\newcommand\iso{\cong}

\newcommand\integers{\ZZ}

\newcommand\calP{{\mathcal P}}

\newcommand\junk[1]{}

\def\lambOneOne{{{
    \begin{tikzpicture}[scale=.1]
    \draw (2,2) -- (0,2) -- (0,0) -- (1,0) -- (1,2);
    \draw (0,1) -- (1,1); 
    \end{tikzpicture}
    }}}
\def\lambOneNaught{{{
    \begin{tikzpicture}[scale=.1]
    \draw (0,1) -- (1,1) -- (1,2) -- (0,2) -- (0,0);
    \end{tikzpicture}
    }}}
\def\lambTwoOneNaughtNaught{{{
    \begin{tikzpicture}[scale=.1]
    \draw (3,4) -- (0,4) -- (0,0); 
    \draw (2,4) -- (2,3) -- (0,3); 
    \draw (1,4) -- (1,2) -- (0,2); 
    \end{tikzpicture}
    }}}
\def\lambTwoOneOneNaught{{{
    \begin{tikzpicture}[scale=.1]
    \draw (3,4) -- (0,4) -- (0,0); 
    \draw (2,4) -- (2,3) -- (0,3); 
    \draw (1,4) -- (1,1) -- (0,1); 
    \draw (1,2) -- (0,2); 
    \end{tikzpicture}
    }}}
\def\lambOneOneOneNaught{{{
    \begin{tikzpicture}[scale=.1]
    \draw (3,4) -- (0,4) -- (0,0); 
    \draw (1,4) -- (1,1) -- (0,1); 
    \draw (1,3) -- (0,3); 
    \draw (1,2) -- (0,2); 
    \end{tikzpicture}
    }}}
\def\lambOneOneNaughtNaught{{{
    \begin{tikzpicture}[scale=.1]
    \draw (3,4) -- (0,4) -- (0,0); 
    \draw (1,4) -- (1,2) -- (0,2); 
    \draw (1,3) -- (0,3); 
    \end{tikzpicture}
    }}}

\newcommand\complexes{{\mathbb C}}
\newcommand\comment[1]{{\bf *** #1 *** }}
\newcommand\Gr[2]{Gr_{#1}(\complexes^{#2})}

\begin{document}

\title{A $K_T$-deformation of the ring of symmetric functions}
\author{Allen Knutson}
\address{Department of Mathematics, Cornell University, Ithaca, NY 14853 USA}
\email{allenk@math.cornell.edu}
\author{Mathias Lederer}
\email{mathias.lederer@uibk.ac.at}
\address{Department of Mathematics \\ 
  University of Innsbruck \\ 
  Technikerstrasse 13 \\
  A-6020 Innsbruck \\ 
  Austria}
\thanks{The first author was supported by NSF grant DMS-0902296.  
  The second author was supported by Marie Curie International
  Outgoing Fellowship 2009-254577 of the EU Seventh Framework
  Program.}  
\date{March 9, 2015}

\begin{abstract}
  The ring of symmetric functions (or a biRees algebra thereof)
  can be implemented in the homology of $\coprod_{a,b} \Gr{a}{a+b}$,
  the multiplicative structure being defined from the ``direct sum'' map. 
  There is a natural circle action (simultaneously on all
  Grassmannians) under which each direct sum map is equivariant. 
  Upon replacing usual homology by equivariant $K$-homology, we obtain
  a $2$-parameter deformation of the ring of symmetric functions.
  
  This ring has a module basis given by Schubert classes $[X^\lambda]$. 
  Geometric considerations show that multiplication of Schubert classes 
  has positive coefficients, in an appropriate sense.
  In this paper we give manifestly positive formul\ae\ for these 
  coefficients: they count numbers of ``DS pipe dreams'' with prescribed 
  edge labelings. \junk{  
  The notion of DS pipe dreams is a generalization of puzzles computing Littlewood-Richardson rules
  due to \fixit{you guys}. 
  Both are tilings of certain regions with prescribed edge labels using certain tiles. 
  However, the original puzzles only used three different types of tiles; 
  DS pipe dreams use no less than 82 of them. 
}
\end{abstract}

\maketitle

{\footnotesize \tableofcontents}

\section{Introduction, and statement of results}

\subsection{The homology ring}

Consider the \defn{direct sum} map on Grassmannians
$$ (V,W) \mapsto V\oplus W, \qquad 
\Gr{a}{a+b} \times \Gr{c}{c+d} \to \Gr{a+c}{a+b+d+c} $$
inducing the map on homology
$$ H_*(\Gr{a}{a+b}) \otimes H_*(\Gr{c}{c+d}) \to H_*(\Gr{a+c}{a+b+d+c}). $$
We can put these maps on homology together to make a trigraded
commutative associative ring:
$$ R^H := \bigoplus_{a,b,*} H_*(\Gr{a}{a+b}) $$
The following is well-known, if not usually stated exactly this way:

\begin{Theorem*}
  Let $Symm$ be the (singly-graded) ring of symmetric functions, 
  with basis of Schur functions indexed by partitions. Let $Symm_{a,b}$ be
  the subspace linearly spanned by the partitions that fit inside a rectangle
  of height $a$ and width $b$. Then the biRees algebra 
  $\bigoplus_{a,b} Symm_{a,b} s^a t^b \leq Symm[s,t]$ is isomorphic to $R^H$,
  where the isomorphism takes each Schur function to the Schubert class
  indexed by the same partition.

  Equivalently, 
  $R^H \bigg/ \left\langle [\Gr{0}{0+1}]-1, [\Gr{1}{1+0}]-1\right\rangle 
  \iso Symm$,
  where the quotient lets one forget the ambient box containing
  the partition.
\end{Theorem*}

In particular, the structure constants in these two rings-with-bases
are computable by the same rule, the Littlewood-Richardson rule.

\subsection{A two-parameter deformation}

Since the direct sum map is equivariant with respect to $GL(a+b)\times GL(c+d)$,
we might hope to extend the ring to the direct sum of the {\em equivariant}
homologies. For that to work, we'd need one group acting on all the
Grassmannians, and the only candidate is $T^\infty \times T^\infty$.
Unfortunately, the reindexing\footnote{%
  It is actually possible to keep this huge action, at the cost of
  making the ring be noncommutative and the base ring $H^*_T$ not be central.
  We did not pursue this, though the geometric techniques used in
  this paper should generalize to that case.}
of the coordinates involved in the direct sum cuts this down to a
single circle action:

\begin{Theorem}
  Let $S$ denote a circle acting on each $\complexes^{a+b}$ with
  weight $1$ on the first $b$ coordinates, weight $0$ on the last $a$,
  and thereby on $\Gr{a}{a+b}$. 
  Then the direct sum map is $S$-equivariant,
  and the induced ring structure on 
  $$ R^{H^S} := \bigoplus_{a,b,*} H^S_*(\Gr{a}{a+b}) $$
  is again commutative associative and trigraded, with a basis (now over
  $H^*_S(pt) \iso \integers[t]$) given by Schubert classes.

  There is a further deformation available (spoiling the $*$-grading) 
  to the sum 
  $$ R^{K^S} := \bigoplus_{a,b} K^S_0(\Gr{a}{a+b}) $$
  of the $S$-equivariant $K$-homology groups, again defining a 
  commutative associative ring (now over $K_S(pt) \iso \integers[\exp(\pm t)]$).
\end{Theorem}

Since $Symm$ is a polynomial ring (i.e. free), any deformation of its
ring structure must be trivializable. Hence there must exist 
deformations $S_\lambda(t)$ of Schur functions, whose multiplication
has the same structure constants as in $R^{H^S}$. 
We did not succeed in trivializing the family and finding such deformations.
Many other deformations of the ring-with-basis of symmetric functions
have been studied (e.g. Hall-Littlewood polynomials, Jack polynomials, 
Macdonald polynomials) and while it is hard to be exhaustive, the two
deformations studied here seem to be new.

It is easy to show for geometric reasons that the structure
constants of these deformed algebras should be nonnegative in whatever
sense appropriate to the cohomology theory by \cite{Kleiman} for
$R^H$, \cite{Graham} for $R^{H^S}$, \cite{Brion} for $R^K$, and
\cite{AGM} for $R^{K^S}$.  The nontrivial contents of the paper are
appropriately positive formul\ae\ for these structure constants.

We now go into detail about the bases and coefficient rings.
Since coordinates we use in this paper always come in two blocks, 
the first of of size $b$, the second of size $a$, 
we will henceforth be writing $b+a$ rather than $a+b$. 
In particular, the direct sum map is defined on $\Gr{a}{b+a}\times\Gr{c}{d+c}$. 
We will see later why we write its range as $\Gr{a+c}{b+d+c+a}$. 

If $\lambda$ is a bit string with content $0^b 1^a$, let $M$ be an
$a\times(b+a)$ matrix of rank $a$ whose columns are $0$ where $\lambda$ is $0$,
and let $X^\lambda$ denote the \defn{Schubert variety}\footnote{%
  If this paper were based on cohomology it would be more natural
  to call this an {\em opposite} Schubert variety.}
$$ X^\lambda := GL(a) \big\backslash \overline{GL(a)\ M\ B_{b+a}} 
        \quad\subseteq \Gr{a}{b+a} $$
where $B_{b+a}$ is the upper triangular matrices of size $b+a$,
the closure is inside matrices of full rank $a$,
and the identification of $X^\lambda$ with a subvariety of the Grassmannian
is by taking a matrix to its row span.

The dimension of $X^\lambda$ is the number of inversions ($1$s occurring
somewhere before $0$s) in $\lambda$. So $X^\lambda$ is a point when
$\lambda = 000\ldots 0111\ldots 1$, in which case $M$ can be taken to
be the identity matrix in the {\em last} $a$ columns. 
%(This is why we keep referring to ``$b+a$'' rather than ``$a+b$''.)

The Schubert classes $[X^\lambda]$ of partitions $\lambda \subseteq a \times b$ 
(the box of height $a$ and width $b$)
form bases of homology modules in various homology theories. 
Four homology theories are of interest to us, 
\begin{itemize}
  \item $H_\star(\Gr{a}{a+b})$ and  $K_0(\Gr{a}{a+b})$, 
    modules over $H_\star(\pt) \iso K_\star(\pt) \iso \Z$, 
  \item $H_\star^S(\Gr{a}{a+b})$, a module over $H_\star^S(\pt) \iso \Z[t]$, and
  \item $K_0^S(\Gr{a}{a+b})$, a module over $K_0^S(\pt) \iso \Z[\exp(\pm t)]$. 
\end{itemize}
Letting both $\lambda$ and their ambient boxes $a \times b$ run, we get bases of the rings $R^H$, $R^{H^S}$ and $R^{K^S}$
as modules over $H_\star(\pt)$, $\Z[t]$ and $\Z[\exp(\pm t)]$, respectively. 

\subsection{Pipe dreams for $R^{H^S}$}

The datum of a partition $\lambda \subseteq a \times b$ translates
into bit string with content $0^b 1^a$ by going from the upper right
corner of $a \times b$ to its lower left corner along the boundary of
$\lambda$ and writing down a $0$ for each horizontal edge and a $1$ for
each vertical edge passed along the way.  Throughout the article, we
move freely back and forth between partitions and bit strings.

Our formul\ae\ for the structure constants of rings-with-bases $R^H$, $R^{H^S}$ and $R^{K^S}$ 
will be sums over certain square tilings, all edges labeled. 
The \defn{lower tiles} are the ones shown in figure \ref{fig:lowerTiles}
with $a,b$ in the set of edge \defn{labels $0,1,R,Q$}, 
and the \defn{upper tiles} are left-right flipped (or upside-down!) versions, 
with the roles of $0$ and $1$ flipped. 
In \cite{K} we had many other letters available as labels but we won't
need them here.

\begin{center}
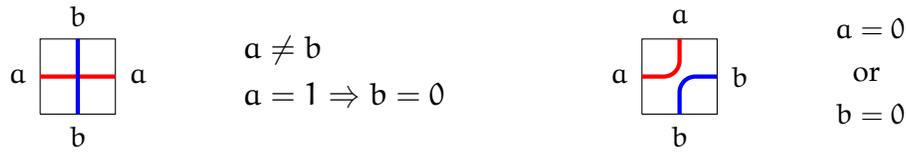
\begin{figure}[ht]
  \begin{tikzpicture}[scale=1]
    % left hand tile
    \draw (0,0) -- (0,1) -- (1,1) -- (1,0) -- (0,0);
    \foreach \x in {0,1.6}
      \draw (\x-.3,.5) node[] {\small$a$};
    \foreach \x in {0,1.6}
      \draw (.5,\x-.3) node[] {\small$b$};
    \draw[Red,ultra thick,rounded corners=6] (0,.5) -- (1,.5);
    \draw[Blue,ultra thick,rounded corners=6] (.5,0) -- (.5,1);
    \draw (4,.5) node[] {$\begin{aligned} a & \neq b \\ a & = 1 \Rightarrow b = 0\end{aligned}$};
    % right hand tile
    \draw (8,0) -- (8,1) -- (9,1) -- (9,0) -- (8,0);
    \draw (8-.3,.5) node[] {\small$a$};
    \draw (8-.3+1.6,.5) node[] {\small$b$};
    \draw (8.5,-.3) node[] {\small$b$};
    \draw (8.5,-.3+1.6) node[] {\small$a$};
    \draw[Red,ultra thick,rounded corners=6] (8,.5) -- (8.5,.5) -- (8.5,1);
    \draw[Blue,ultra thick,rounded corners=6] (8.5,0) -- (8.5,.5) -- (9,.5);
    \draw (11,.5) node[] {\small$\begin{aligned} a & = 0 \\ & \text{or} \\ b & = 0\end{aligned}$};
  \end{tikzpicture}
  \caption{Lower tiles, the ``crossing'' and ``elbows''. 
  If $a=b=0$ in an elbows tile, it is the \defn{equivariant} tile.}
  \label{fig:lowerTiles}
\end{figure}
\end{center}

Define a \defn{DS pipe dream} (for direct sum\footnote{%
  Pipe dreams come up in seemingly unrelated contexts, such as those
  in \cite{KM} and \cite{K}, so it seems safest
  to include a modifier. The ones here are of closely related to those
  in \cite{K}, as the proofs will show.})
of partitions $\lambda \subseteq a \times b$ and $\mu \subseteq c \times d$ as a tiling of the region from figure \ref{fig:DSPipeDream}, 
with the following restrictions on the tiles and the labels on the West, South and East boundary: 

\begin{center}
\begin{figure}[ht]
  \begin{tikzpicture}[scale=.6]
  \fill[SkyBlue] (0,9) rectangle (9,4);
  \fill[SkyBlue] (9,4) rectangle (12,0);
  \draw (0,9) -- (0,3) -- (1,3) -- (1,2) -- (2,2) -- (2,1) -- (3,1) -- (3,0) -- (12,0) -- (12,4) -- (13,4) -- (13,5) -- (14,5) -- (14,6) -- (15,6) -- (15,7) -- (16,7) -- (16,8) -- (17,8) -- (17,9) -- (0,9);
  \draw [densely dashed] (4,0) -- (4,9);
  \draw [densely dashed] (12,4) -- (12,9);
  \foreach \x in {0,1,2,3}
    \draw (\x,3.5-\x) node[] {$0$};
  \foreach \x in {0,1,2,3}
    \draw (\x+.5,3-\x) node[] {$R$};
  \foreach \x in {0,1}
    \draw (\x+12.5,4+\x) node[] {$Q$};
  \foreach \x in {0,1,2}
    \draw (\x+14.5,6+\x) node[] {$0$};
  \foreach \x in {0,1,2,3,4}
    \draw (\x+13.1,4.5+\x) node[] {$1$};
  \draw (8,-.8) node[] {$\mu$ in $d$ 0s and $c$ 1s};
  \draw (12.7,2) node[] {$\lambda_1$};
  \draw (-.6,6.5) node[] {$\lambda_2$};
  \draw (2,9.7) node[] {$b$};
  \draw (6.5,9.7) node[] {$d$};
  \draw (10.5,9.7) node[] {$c$};
  \draw (14.5,9.7) node[] {$a$};
    \end{tikzpicture}
  \caption{The region to be filled by DS pipe dreams}
  \label{fig:DSPipeDream}
\end{figure}
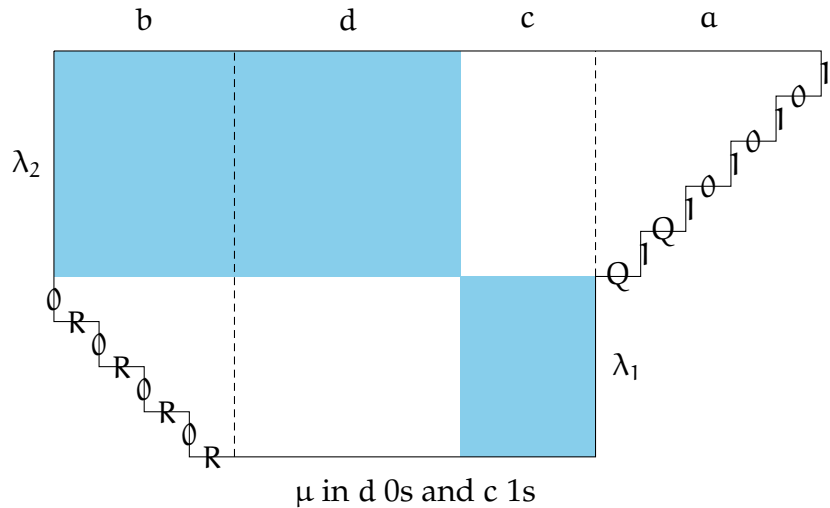
\end{center}

\begin{itemize}
\item The tiles in the lower and upper half are lower and upper tiles,
  respectively. \\ (In fact the boundary conditions ensure that no lower
  tile used will involve $Q$.)
  \item Across the South boundary is the bit string of partition $\mu$
    read left to right, a string of $d$ $0$s and $c$ $1$s.
  \item Write $\lambda$ as a bit string of $b$ $0$s and $a$ $1$s. Then
    \comment{recheck}
    \begin{itemize}
    \item $\lambda_1$ (on the East side) is the first $b$ letters of
      $\lambda$, read bottom to top, with the $0$s turned into $R$s and
      the $1$s turned to $0$s, and
    \item $\lambda_2$ is the last $a$ letters of $\lambda$, read top to bottom, with
      the $0$s turned into $R$s and the $1$s turned to $Q$s
    \end{itemize}
  \item There are as many $Q$s on horizontal edges below the ``$a$'' as there are $Q$s in $\lambda_2$. 
  They are flush left, and the horizontal edges to their right carry 0s. 
  The vertical edges there carry 1s.
  \item Horizontal edges below the ``$b$'' carry $R$s, vertical edges carry 0s.  
  \item Any equivariant tiles only appear in the shaded regions.  
\end{itemize}

We will present a more mnemonic picture explaining the labels of the region at the very end of the paper. 
To a DS pipe dream $P$ of partitions $\lambda \subseteq a \times b$ 
and $\mu \subseteq c \times d$, 
we associate a partition $\nu(P)$ depending on only the North labels. 
The restrictions we impose on the tiles and on the West, South and East labels imply that 
the only letters appearing on the North labels of the region are $0$ and 1. 
This makes for a word of length $b+d+c+a$ in $0$ and 1, which translates into a partition $\nu(P)$.

Define the \defn{$H^S$-weight} $\weight'_{S}(P) \in H_\star^S(\pt) = \Z[t]$
of a DS pipe dream $P$ as $t^{\#\{ \text{equivariant tiles}\}}$.
(This will be the leading term, as $t\to 1$, of the $K^S$-weight coming later.)

\begin{Theorem}
  \begin{enumerate}
  \item   As elements of $H_\star(\Gr{a+c}{b+d+c+a})$, the expansion of $[X^\lambda]\cdot[X^\mu] = [X^\lambda \oplus X^\mu]$ 
  in the $H_\star(\pt)$-basis of Schubert classes is 
  \[
    [X^\lambda \oplus X^\mu] = \sum_{\substack{\text{DS-pipe dreams } P \text{ of partitions } \lambda \text{ and }\mu \\
    \text{ having no equivariant tiles}}} [X^{\nu(P)}]
    = \sum_\nu \# \left\{ \begin{array}{c} \text{DS-pipe dreams } P: \\ \nu(P) = P, \\ P \text{ has no equivariant tiles} \end{array} \right\} [X^\nu].
  \]
  \item  As elements of $H_\star^S(\Gr{a+c}{b+d+c+a})$, the expansion 
  in the $H_\star^S(\pt)$-basis is 
  \[
    [X^\lambda \oplus X^\mu] 
    = \sum_{P \text{ of partitions } \lambda \text{ and }\mu \\
    \text{with no strict $K$-tiles}} \weight'_S(P)\ [X^{\nu(P)}]
  \]
  Specializing $t$ to $0$ recovers the formula from (1).
\end{enumerate}
\end{Theorem}

The first formula in this theorem is of course just another rule for the
standard Littlewood-Richardson number, but we include it to relate it to
the $H^S$-deformation in the second half.

\begin{example}\label{example:H}
  When expanding the class of $X^\lambOneOne \oplus X^\lambOneNaught$ in the basis of Schubert classes, 
  we have to fill the region 
  
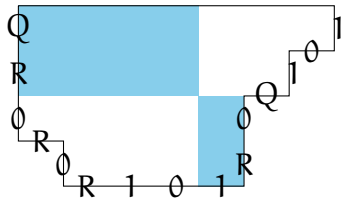
\begin{center}
  \begin{tikzpicture}[scale=.6]
  \fill[SkyBlue] (0,4) rectangle (4,2);
  \fill[SkyBlue] (4,2) rectangle (5,0);
  \draw (0,4) -- (0,1) -- (1,1) -- (1,0) -- (5,0) -- (5,2) -- (6,2) -- (6,3) -- (7,3) -- (7,4) -- (0,4);
  \foreach \x in {0,1}
    \draw (\x,1.5-\x) node[] {$0$};
  \foreach \x in {0,2}
    \draw (\x+2.5,0) node[] {$1$};
  \foreach \x in {0,1}
    \draw (\x+6.1,2.5+\x) node[] {$1$};
  \foreach \x in {0,1}
    \draw (0.5 + \x,1-\x) node[] {$R$};
  \draw (0,3.5) node[] {$Q$};
  \draw (0,2.5) node[] {$R$};
  \draw (5,0.5) node[] {$R$};
  \draw (5,1.5) node[] {$0$};
  \draw (3.5,0) node[] {$0$};
  \draw (5.5,2) node[] {$Q$};
  \draw (6.5,3) node[] {$0$};
  \end{tikzpicture}
  \end{center}
  with tiles. We find two pipe dreams using no equivariant tiles, 
  
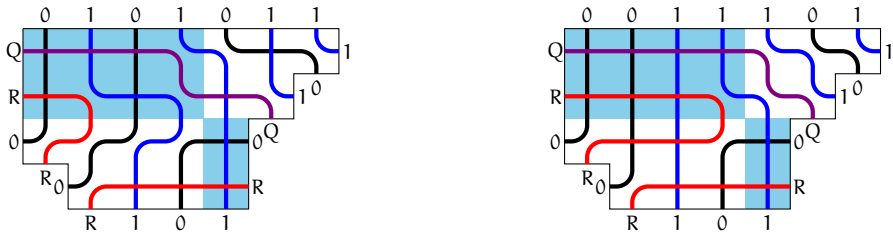
\begin{center}
  \begin{tikzpicture}[scale=.6]
  \foreach \xoffset in {0,12}
  {
    \begin{scope}[shift={(\xoffset,0)}]
  \fill[SkyBlue] (0,4) rectangle (4,2);
  \fill[SkyBlue] (4,2) rectangle (5,0);
  \draw (0,4) -- (0,1) -- (1,1) -- (1,0) -- (5,0) -- (5,2) -- (6,2) -- (6,3) -- (7,3) -- (7,4) -- (0,4);
  \foreach \x in {0,1}
    \draw (\x-.2,1.5-\x) node[] {\tiny$0$};
  \foreach \x in {0,2}
    \draw (\x+2.5,-.3) node[] {\tiny$1$};
  \foreach \x in {0,1}
    \draw (\x+6.2,2.5+\x) node[] {\tiny$1$};
  \foreach \x in {0,1}
    \draw (0.5 + \x,1-\x-.3) node[] {\tiny$R$};
  \draw (0-.2,3.5) node[] {\tiny$Q$};
  \draw (0-.2,2.5) node[] {\tiny$R$};
  \draw (5+.2,0.5) node[] {\tiny$R$};
  \draw (5+.2,1.5) node[] {\tiny$0$};
  \draw (3.5,0-.3) node[] {\tiny$0$};
  \draw (5.5,2-.3) node[] {\tiny$Q$};
  \draw (6.5,3-.3) node[] {\tiny$0$};
    \end{scope}
  }
  % 1st set of pipes
  \draw[ultra thick,rounded corners=6] (4.5,4) -- (4.5,3.5) -- (6.5,3.5) -- (6.5,3);
  \draw[ultra thick,rounded corners=6] (2.5,4) -- (2.5,1.5) -- (1.5,1.5) -- (1.5,.5) -- (1,.5);
  \draw[ultra thick,rounded corners=6] (.5,4) -- (.5,1.5) -- (0,1.5);
  \draw[ultra thick,rounded corners=6] (3.5,0) -- (3.5,1.5) -- (5,1.5);
  \draw[Blue,ultra thick,rounded corners=6] (6.5,4) -- (6.5,3.5) -- (7,3.5);
  \draw[Blue,ultra thick,rounded corners=6] (5.5,4) -- (5.5,2.5) -- (6,2.5);
  \draw[Blue,ultra thick,rounded corners=6] (3.5,4) -- (3.5,3.5) -- (4.5,3.5) -- (4.5,0);
  \draw[Blue,ultra thick,rounded corners=6] (1.5,4) -- (1.5,2.5) -- (3.5,2.5) -- (3.5,1.5) -- (2.5,1.5) -- (2.5,0);
  \draw[Red,ultra thick,rounded corners=6] (0,2.5) -- (1.5,2.5) -- (1.5,1.5) -- (.5,1.5) -- (.5,1);
  \draw[Red,ultra thick,rounded corners=6] (1.5,0) -- (1.5,.5) -- (5,.5); 
  \draw[Purple,ultra thick,rounded corners=6] (0,3.5) -- (3.5,3.5) -- (3.5,2.5) -- (5.5,2.5) -- (5.5,2);
  \foreach \x in {0,2,4}
    \draw (\x+.5,4.3) node[] {\tiny$0$};
  \foreach \x in {1,3,5,6}
    \draw (\x+.5,4.3) node[] {\tiny$1$};
  % 2nd set of pipes 
  \draw[ultra thick,rounded corners=6] (17.5,4) -- (17.5,3.5) -- (18.5,3.5) -- (18.5,3);
  \draw[ultra thick,rounded corners=6] (13.5,4) -- (13.5,.5) -- (13,.5);
  \draw[ultra thick,rounded corners=6] (.5+12,4) -- (.5+12,1.5) -- (0+12,1.5);
  \draw[ultra thick,rounded corners=6] (3.5+12,0) -- (3.5+12,1.5) -- (5+12,1.5);
  \draw[Blue,ultra thick,rounded corners=6] (6.5+12,4) -- (6.5+12,3.5) -- (7+12,3.5);
  \draw[Blue,ultra thick,rounded corners=6] (16.5,4) -- (16.5,3.5) -- (17.5,3.5) -- (17.5,2.5) -- (18,2.5);
  \draw[Blue,ultra thick,rounded corners=6] (15.5,4) -- (15.5,2.5) -- (16.5,2.5) -- (16.5,0);
  \draw[Blue,ultra thick,rounded corners=6] (14.5,4) -- (14.5,0);
  \draw[Red,ultra thick,rounded corners=6] (12,2.5) -- (15.5,2.5) -- (15.5,1.5) -- (12.5,1.5) -- (12.5,1);
  \draw[Red,ultra thick,rounded corners=6] (1.5+12,0) -- (1.5+12,.5) -- (5+12,.5); 
  \draw[Purple,ultra thick,rounded corners=6] (12,3.5) -- (16.5,3.5) -- (16.5,2.5) -- (17.5,2.5) -- (17.5,2);
  \foreach \x in {0,1,5}
    \draw (\x+12.5,4.3) node[] {\tiny$0$};
  \foreach \x in {2,3,4,6}
    \draw (\x+12.5,4.3) node[] {\tiny$1$};
  \end{tikzpicture}
  \end{center}  
  which says that 
  \[
    \left[ X^\lambOneOne \oplus X^\lambOneNaught \right] = \left[ X^\lambTwoOneNaughtNaught \right] + \left[ X^\lambOneOneOneNaught \right]
  \]
  as elements of $H_\star(\Gr{4}{7})$. 
  We find one pipe dream using an equivariant tile, a 1-1 elbow at $(4,2)$. 
  
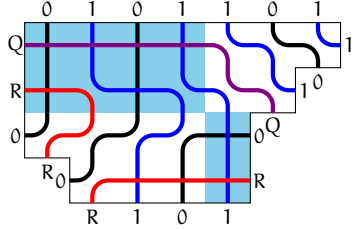
\begin{center}
  \begin{tikzpicture}[scale=.6]
  \fill[SkyBlue] (0,4) rectangle (4,2);
  \fill[SkyBlue] (4,2) rectangle (5,0);
  \draw (0,4) -- (0,1) -- (1,1) -- (1,0) -- (5,0) -- (5,2) -- (6,2) -- (6,3) -- (7,3) -- (7,4) -- (0,4);
  \foreach \x in {0,1}
    \draw (\x-.2,1.5-\x) node[] {\tiny$0$};
  \foreach \x in {0,2}
    \draw (\x+2.5,-.3) node[] {\tiny$1$};
  \foreach \x in {0,1}
    \draw (\x+6.2,2.5+\x) node[] {\tiny$1$};
  \foreach \x in {0,1}
    \draw (0.5 + \x,1-\x-.3) node[] {\tiny$R$};
  \draw (0-.2,3.5) node[] {\tiny$Q$};
  \draw (0-.2,2.5) node[] {\tiny$R$};
  \draw (5+.2,0.5) node[] {\tiny$R$};
  \draw (5+.2,1.5) node[] {\tiny$0$};
  \draw (3.5,0-.3) node[] {\tiny$0$};
  \draw (5.5,2-.3) node[] {\tiny$Q$};
  \draw (6.5,3-.3) node[] {\tiny$0$};
  \draw[ultra thick,rounded corners=6] (5.5,4) -- (5.5,3.5) -- (6.5,3.5) -- (6.5,3);
  \draw[ultra thick,rounded corners=6] (2.5,4) -- (2.5,1.5) -- (1.5,1.5) -- (1.5,.5) -- (1,.5);
  \draw[ultra thick,rounded corners=6] (.5,4) -- (.5,1.5) -- (0,1.5);
  \draw[ultra thick,rounded corners=6] (3.5,0) -- (3.5,1.5) -- (5,1.5);
  \draw[Blue,ultra thick,rounded corners=6] (6.5,4) -- (6.5,3.5) -- (7,3.5);
  \draw[Blue,ultra thick,rounded corners=6] (4.5,4) -- (4.5,3.5) -- (5.5,3.5) -- (5.5,2.5) -- (6,2.5);
  \draw[Blue,ultra thick,rounded corners=6] (3.5,4) -- (3.5,2.5) -- (4.5,2.5) -- (4.5,0);
  \draw[Blue,ultra thick,rounded corners=6] (1.5,4) -- (1.5,2.5) -- (3.5,2.5) -- (3.5,1.5) -- (2.5,1.5) -- (2.5,0);
  \draw[Red,ultra thick,rounded corners=6] (0,2.5) -- (1.5,2.5) -- (1.5,1.5) -- (.5,1.5) -- (.5,1);
  \draw[Red,ultra thick,rounded corners=6] (1.5,0) -- (1.5,.5) -- (5,.5); 
  \draw[Purple,ultra thick,rounded corners=6] (0,3.5) -- (4.5,3.5) -- (4.5,2.5) -- (5.5,2.5) -- (5.5,2);
  \foreach \x in {0,2,5}
    \draw (\x+.5,4.3) node[] {\tiny$0$};
  \foreach \x in {1,3,4,6}
    \draw (\x+.5,4.3) node[] {\tiny$1$};
  \end{tikzpicture}
  \end{center}  
  Therefore, 
  \[
    \left[ X^\lambOneOne \oplus X^\lambOneNaught \right] = \left[ X^\lambTwoOneNaughtNaught \right] + \left[ X^\lambOneOneOneNaught \right]
    + t \left[ X^\lambTwoOneOneNaught \right]
  \]
  as elements of $H_\star^S(\Gr{4}{7})$. 
  
  However, since DS pipe dreams are far from symmetric in $\lambda$ and $\mu$, 
  we also compute the class of $X^\lambOneNaught \oplus X^\lambOneOne$ to illustrate commutativity. 
  When expanding the class of $X^\lambOneOne \oplus X^\lambOneNaught$ in the basis of Schubert classes, 
  we have to fill the region 
  
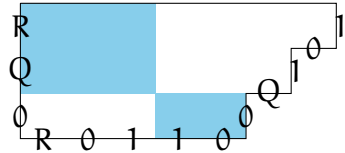
\begin{center}
  \begin{tikzpicture}[scale=.6]
  \fill[SkyBlue] (0,3) rectangle (3,1);
  \fill[SkyBlue] (3,1) rectangle (5,0);
  \draw (0,3) -- (0,0) -- (5,0) -- (5,1) -- (6,1) -- (6,2) -- (7,2) -- (7,3) -- (0,3);
  \draw (0,.5) node[] {$0$};
  \draw (0,1.5) node[] {$Q$};
  \draw (0,2.5) node[] {$R$};
  \draw (.5,0) node[] {$R$};
  \foreach \x in {1,4}
    \draw (\x+.5,0) node[] {$0$};
  \foreach \x in {2,3}
    \draw (\x+.5,0) node[] {$1$};
  \draw (5,.5) node[] {$0$};
  \draw (5.5,1) node[] {$Q$};
  \draw (6.5,2) node[] {$0$};
  \foreach \x in {1,2}
    \draw (5.1+\x,.5+\x) node[] {$1$};
  \end{tikzpicture}
  \end{center}
  We find two pipe dreams using no equivariant tiles, 
  
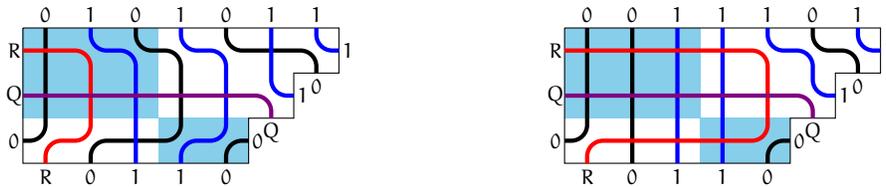
\begin{center}
  \begin{tikzpicture}[scale=.6]
  \foreach \xoffset in {0,12}
  {
    \begin{scope}[shift={(\xoffset,0)}]
  \fill[SkyBlue] (0,3) rectangle (3,1);
  \fill[SkyBlue] (3,1) rectangle (5,0);
  \draw (0,3) -- (0,0) -- (5,0) -- (5,1) -- (6,1) -- (6,2) -- (7,2) -- (7,3) -- (0,3);
  \draw (-.2,.5) node[] {\tiny$0$};
  \foreach \x in {0,1}
    \draw (\x+2.5,-.3) node[] {\tiny$1$};
  \foreach \x in {0,1}
    \draw (\x+6.2,1.5+\x) node[] {\tiny$1$};
  \draw (0.5,-.3) node[] {\tiny$R$};
  \draw (0-.2,1.5) node[] {\tiny$Q$};
  \draw (0-.2,2.5) node[] {\tiny$R$};
  \draw (5+.2,0.5) node[] {\tiny$0$};
  \foreach \x in {0,3}
  \draw (1.5 + \x,-.3) node[] {\tiny$0$};
  \draw (5.5,1-.3) node[] {\tiny$Q$};
  \draw (6.5,2-.3) node[] {\tiny$0$};
  \end{scope}
  }
  % 1st set of pipes
  \draw[ultra thick,rounded corners=6] (.5,3) -- (.5,.5) -- (0,.5);
  \draw[ultra thick,rounded corners=6] (2.5,3) -- (2.5,2.5) -- (3.5,2.5) -- (3.5,.5) -- (1.5,.5) -- (1.5,0);
  \draw[ultra thick,rounded corners=6] (4.5,3) -- (4.5,2.5) -- (6.5,2.5) -- (6.5,2); 
  \draw[ultra thick,rounded corners=6] (5,.5) -- (4.5,.5) -- (4.5,0);
  \draw[Blue,ultra thick,rounded corners=6] (1.5,3) -- (1.5,2.5) -- (2.5,2.5) -- (2.5,0);
  \draw[Blue,ultra thick,rounded corners=6] (3.5,3) -- (3.5,2.5) -- (4.5,2.5) -- (4.5,.5) -- (3.5,.5) -- (3.5,0);
  \draw[Blue,ultra thick,rounded corners=6] (5.5,3) -- (5.5,1.5) -- (6,1.5); 
  \draw[Blue,ultra thick,rounded corners=6] (6.5,3) -- (6.5,2.5) -- (7,2.5); 
  \draw[Red,ultra thick,rounded corners=6] (0,2.5) -- (1.5,2.5) -- (1.5,.5) -- (.5,.5) -- (.5,0);
  \draw[Purple,ultra thick,rounded corners=6] (0,1.5) -- (5.5,1.5) -- (5.5,1);
  \foreach \x in {0,2,4}
    \draw (\x+.5,3.3) node[] {\tiny$0$};
  \foreach \x in {1,3,5,6}
    \draw (\x+.5,3.3) node[] {\tiny$1$};
  % 2nd set of pipes
  \draw[ultra thick,rounded corners=6] (12.5,3) -- (12.5,.5) -- (12,.5);
  \draw[ultra thick,rounded corners=6] (13.5,3) -- (13.5,0);
  \draw[ultra thick,rounded corners=6] (17.5,3) -- (17.5,2.5) -- (18.5,2.5) -- (18.5,2);
  \draw[ultra thick,rounded corners=6] (17,.5) -- (16.5,.5) -- (16.5,0);
  \draw[Blue,ultra thick,rounded corners=6] (14.5,3) -- (14.5,0);
  \draw[Blue,ultra thick,rounded corners=6] (15.5,3) -- (15.5,0);
  \draw[Blue,ultra thick,rounded corners=6] (16.5,3) -- (16.5,2.5) -- (17.5,2.5) -- (17.5,1.5) -- (18,1.5);
  \draw[Blue,ultra thick,rounded corners=6] (18.5,3) -- (18.5,2.5) -- (19,2.5);
  \draw[Red,ultra thick,rounded corners=6] (12,2.5) -- (16.5,2.5) -- (16.5,.5) -- (12.5,.5) -- (12.5,0);
  \draw[Purple,ultra thick,rounded corners=6] (12,1.5) -- (17.5,1.5) -- (17.5,1);
  \foreach \x in {0,1,5}
    \draw (\x+12.5,3.3) node[] {\tiny$0$};
  \foreach \x in {2,3,4,6}
    \draw (\x+12.5,3.3) node[] {\tiny$1$};
  \end{tikzpicture}
  \end{center}
  and one pipe dream using an equivariant tile, a 1-1 elbow at $(5,3)$. 
  
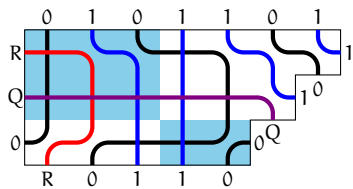
\begin{center}
  \begin{tikzpicture}[scale=.6]
  \fill[SkyBlue] (0,3) rectangle (3,1);
  \fill[SkyBlue] (3,1) rectangle (5,0);
  \draw (0,3) -- (0,0) -- (5,0) -- (5,1) -- (6,1) -- (6,2) -- (7,2) -- (7,3) -- (0,3);
  \draw (-.2,.5) node[] {\tiny$0$};
  \foreach \x in {0,1}
    \draw (\x+2.5,-.3) node[] {\tiny$1$};
  \foreach \x in {0,1}
    \draw (\x+6.2,1.5+\x) node[] {\tiny$1$};
  \draw (0.5,-.3) node[] {\tiny$R$};
  \draw (0-.2,1.5) node[] {\tiny$Q$};
  \draw (0-.2,2.5) node[] {\tiny$R$};
  \draw (5+.2,0.5) node[] {\tiny$0$};
  \foreach \x in {0,3}
  \draw (1.5 + \x,-.3) node[] {\tiny$0$};
  \draw (5.5,1-.3) node[] {\tiny$Q$};
  \draw (6.5,2-.3) node[] {\tiny$0$};
  \draw[ultra thick,rounded corners=6] (.5,3) -- (.5,.5) -- (0,.5);
  \draw[ultra thick,rounded corners=6] (2.5,3) -- (2.5,2.5) -- (4.5,2.5) -- (4.5,.5) -- (1.5,.5) -- (1.5,0);
  \draw[ultra thick,rounded corners=6] (5.5,3) -- (5.5,2.5) -- (6.5,2.5) -- (6.5,2);
  \draw[ultra thick,rounded corners=6] (5,.5) -- (4.5,.5) -- (4.5,0);
  \draw[Blue,ultra thick,rounded corners=6] (1.5,3) -- (1.5,2.5) -- (2.5,2.5) -- (2.5,0);
  \draw[Blue,ultra thick,rounded corners=6] (3.5,3) -- (3.5,0);
  \draw[Blue,ultra thick,rounded corners=6] (4.5,3) -- (4.5,2.5) -- (5.5,2.5) -- (5.5,1.5) -- (6,1.5); 
  \draw[Blue,ultra thick,rounded corners=6] (6.5,3) -- (6.5,2.5) -- (7,2.5); 
  \draw[Red,ultra thick,rounded corners=6] (0,2.5) -- (1.5,2.5) -- (1.5,.5) -- (.5,.5) -- (.5,0);
  \draw[Purple,ultra thick,rounded corners=6] (0,1.5) -- (5.5,1.5) -- (5.5,1);
  \foreach \x in {0,2,5}
    \draw (\x+.5,3.3) node[] {\tiny$0$};
  \foreach \x in {1,3,4,6}
    \draw (\x+.5,3.3) node[] {\tiny$1$};
  \end{tikzpicture}
  \end{center}
  The North labels of these pipe dreams are indeed the same 
  as the North labels of the three pipe dreams for $X^\lambOneOne \oplus X^\lambOneNaught$. 
\end{example}

\subsection{Pipe dreams for $R^{K^S}$}

For handling the structure constants of the ring-with-basis $R^{K^S}$, 
we use more general labels $W$ on the vertical edges of tiles. 
Each $W$ is a word in $\{1, R, Q\}$ (no $0$s), no letters repeating, 
and if it contains $1$ then the $1$ must be at the end, 
so $W \in \{ \emptyset, 1, R, Q, RQ, QR, R1, Q1, RQ1, QR1\}$.
There are four kinds of \defn{lower $K$-tiles}, including the
fundamentally new ``displacer'' tile, as shown in figure \ref{fig:lowerKTiles}. 
As in the $H^S$-case, the crossing tile is subject to the restrictions $W \neq b$ and $W = 1 \Rightarrow b = 0$. 
The empty word can only appear in the fusor and in the displacer tile; 
a fusor tile with $W = \emptyset$ is a usual elbow tile as used in DS pipe dreams. 
If a tile has a word with more than one letter,
call it a \defn{strict $K$-tile}.

\begin{center}
\begin{figure}[ht]
  \begin{tikzpicture}[scale=1]
  \foreach \xoffset in {0,4,8,12}
  {
    \begin{scope}[shift={(\xoffset,0)}]
    \draw (0,0) -- (0,1) -- (1,1) -- (1,0) -- (0,0);
  \end{scope}
  }
  % 1st tile
    \draw (.5,1.8) node[] {\small crossing};
    \foreach \x in {0,1.6}
      \draw (\x-.3,.5) node[] {\small$W$};
    \foreach \x in {0,1.6}
      \draw (.5,\x-.3) node[] {\small$b$};
    \draw[Red,ultra thick] (0,.5) -- (1,.5);
    \draw[Blue,ultra thick] (.5,0) -- (.5,1);
    % 2nd tile
    \draw (4.5,1.8) node[] {\small dot};
    \draw[Blue,ultra thick,rounded corners=6] (4,.5) -- (4.5,.5) -- (4.5,1);
    \draw[Red,ultra thick,rounded corners=6] (4.5,0) -- (4.5,.5) -- (5,.5);
    \draw (4-.3,.5) node[] {\small$0$};
    \draw (4-.3+1.6,.5) node[] {\small$b$};
    \draw (4.5,-.3) node[] {\small$b$};
    \draw (4.5,-.3+1.6) node[] {\small$0$};
    % 3rd tile
  \foreach \xoffset in {8,12}
  {
    \begin{scope}[shift={(\xoffset,0)}]
    \draw[Red,ultra thick,rounded corners=6] (0,.5) -- (.5,.5) -- (.5,1);
    \draw[Blue,ultra thick,rounded corners=6] (.5,0) -- (.5,.5) -- (1,.5);
  \end{scope}
  }
    \draw[ultra thick] (8,.5) -- (8.2,.3);
    \draw (8.5,1.8) node[] {\small fusor};
    \draw (8.2,.17) node[] {\tiny$W$};
    \draw (8-.4,.5) node[] {\small$Wb$};
    \draw (8-.3+1.6,.5) node[] {\small$0$};
    \draw (8.5,-.3) node[] {\small$0$};
    \draw (8.5,-.3+1.6) node[] {\small$b$};
    % 4th tile
    \draw (12.5,1.8) node[] {\small displacer};
    \draw[Red,ultra thick,rounded corners=6] (12.5,1) -- (12.5,.5) -- (13,.5);
    \draw (12-.4,.5) node[] {\small$Wb$};
    \draw (12-.3+1.8,.5) node[] {\small$Wbc$};
    \draw (12.5,-.3) node[] {\small$c$};
    \draw (12.5,-.3+1.6) node[] {\small$b$};
  \end{tikzpicture}
  \caption{Lower $K$-tiles}
  \label{fig:lowerKTiles}
\end{figure}
\end{center}

There are a total of 53 lower $K$-tiles. 
The following list shows all of them in typical positions where alignment happens. 
The first four lines show honest $K$-tiles, 
the rest the tiles for $H^S$. 
Only two tiles appear more than once in this list. 
\begin{center}
  \begin{tikzpicture}[scale=.7]
  % 1st row
  \foreach \xoffset in {0,1}
  {
  \begin{scope}[shift={(\xoffset*2+.5,7)}]
    \draw (0,0) -- (0,1) -- (1,1) -- (1,0) -- (0,0);
  \end{scope}
  }
  \begin{scope}[shift={(.5,7)}]
    \draw[Red,ultra thick,rounded corners=6] (0,.5) -- (.5,.5) -- (.5,1);  
    \draw[Purple,ultra thick,rounded corners=6] (.5,0) -- (.5,.5) -- (1,.5);  
    \draw[Red,ultra thick,rounded corners=6] (1,.5) -- (.5,.5) -- (.5,1);  
    \draw (-.3,.5) node[] {\tiny$R$};
    \draw (.5,1.3) node[] {\tiny$R$};
    \draw (.5,-.3) node[] {\tiny$Q$};
    \draw (1.5,.5) node[] {\tiny$RQ$};
    \draw[Purple,ultra thick,rounded corners=6] (2,.5) -- (2.5,.5) -- (2.5,1);  
    \draw[ultra thick,rounded corners=6] (2.5,0) -- (2.5,.5) -- (3,.5);  
    \draw[Red,ultra thick,rounded corners=6] (2,.5) -- (2.2,.3);  
    \draw (2.5,1.3) node[] {\tiny$Q$};
    \draw (2.5,-.3) node[] {\tiny$0$};
    \draw (3.3,.5) node[] {\tiny$0$};
    \draw (2.2,.17) node[] {\tiny$R$};
  \end{scope}
  \foreach \xoffset in {0,1,2,3,4,5}
  {
  \begin{scope}[shift={(8.5+\xoffset*2,7)}]
    \draw (0,0) -- (0,1) -- (1,1) -- (1,0) -- (0,0);
  \end{scope}
  }
  \begin{scope}[shift={(8.5,7)}]
    \draw[Red,ultra thick,rounded corners=6] (0,.5) -- (.5,.5) -- (.5,1);  
    \draw[Purple,ultra thick,rounded corners=6] (.5,0) -- (.5,.5) -- (1,.5);  
    \draw[Red,ultra thick,rounded corners=6] (1,.5) -- (.5,.5) -- (.5,1);  
    \draw (-.3,.5) node[] {\tiny$R$};
    \draw (.5,1.3) node[] {\tiny$R$};
    \draw (.5,-.3) node[] {\tiny$Q$};
    \foreach \off in {0,1,2,3,4}
      \draw (1.5+\off*2,.5) node[] {\tiny$RQ$};
    \foreach \off in {0,1,2,3}
      \draw[Purple,ultra thick] (2+\off*2,.5) -- (3+\off*2,.5);
    \draw[Purple,ultra thick,rounded corners=6] (10,.5) -- (10.5,.5) -- (10.5,1);  
    \draw[ultra thick,rounded corners=6] (10.5,0) -- (10.5,.5) -- (11,.5);  
    \draw[Red,ultra thick,rounded corners=6] (10,.5) -- (10.2,.3);  
    \draw (10.5,1.3) node[] {\tiny$Q$};
    \draw (10.5,-.3) node[] {\tiny$0$};
    \draw (11.3,.5) node[] {\tiny$0$};
    \draw (10.2,.17) node[] {\tiny$R$};
    \foreach \off in {0}
      {
      \draw[ultra thick] (2.5+\off*10,0) -- (2.5+\off*10,1);
      \draw (2.5+\off*10,-.3) node[] {\tiny$0$};
      \draw (2.5+\off*10,1.3) node[] {\tiny$0$};
      \draw[Blue,ultra thick] (4.5+\off*10,0) -- (4.5+\off*10,1);
      \draw (4.5+\off*10,-.3) node[] {\tiny$1$};
      \draw (4.5+\off*10,1.3) node[] {\tiny$1$};
      \draw[Red,ultra thick] (6.5+\off*10,0) -- (6.5+\off*10,1);
      \draw (6.5+\off*10,-.3) node[] {\tiny$R$};
      \draw (6.5+\off*10,1.3) node[] {\tiny$R$};
      \draw[Purple,ultra thick] (8.5+\off*10,0) -- (8.5+\off*10,1);
      \draw (8.5+\off*10,-.3) node[] {\tiny$Q$};
      \draw (8.5+\off*10,1.3) node[] {\tiny$Q$};
      }
  \end{scope}
  % 2nd row
  \foreach \xoffset in {0,1}
  {
  \begin{scope}[shift={(\xoffset*2+.5,4.5)}]
    \draw (0,0) -- (0,1) -- (1,1) -- (1,0) -- (0,0);
  \end{scope}
  }
  \begin{scope}[shift={(.5,4.5)}]
    \draw[Purple,ultra thick,rounded corners=6] (0,.5) -- (.5,.5) -- (.5,1);  
    \draw[Red,ultra thick,rounded corners=6] (.5,0) -- (.5,.5) -- (1,.5);  
    \draw[Purple,ultra thick,rounded corners=6] (1,.5) -- (.5,.5) -- (.5,1);  
    \draw (-.3,.5) node[] {\tiny$Q$};
    \draw (.5,1.3) node[] {\tiny$Q$};
    \draw (.5,-.3) node[] {\tiny$R$};
    \draw (1.5,.5) node[] {\tiny$QR$};
    \draw[Red,ultra thick,rounded corners=6] (2,.5) -- (2.5,.5) -- (2.5,1);  
    \draw[ultra thick,rounded corners=6] (2.5,0) -- (2.5,.5) -- (3,.5);  
    \draw[Purple,ultra thick,rounded corners=6] (2,.5) -- (2.2,.3);  
    \draw (2.5,1.3) node[] {\tiny$R$};
    \draw (2.5,-.3) node[] {\tiny$0$};
    \draw (3.3,.5) node[] {\tiny$0$};
    \draw (2.2,.17) node[] {\tiny$Q$};
  \end{scope}
  \foreach \xoffset in {0,1,2,3,4,5}
  {
  \begin{scope}[shift={(8.5+\xoffset*2,4.5)}]
    \draw (0,0) -- (0,1) -- (1,1) -- (1,0) -- (0,0);
  \end{scope}
  }
  \begin{scope}[shift={(8.5,4.5)}]
    \draw[Purple,ultra thick,rounded corners=6] (0,.5) -- (.5,.5) -- (.5,1);  
    \draw[Red,ultra thick,rounded corners=6] (.5,0) -- (.5,.5) -- (1,.5);  
    \draw[Purple,ultra thick,rounded corners=6] (1,.5) -- (.5,.5) -- (.5,1);  
    \draw (-.3,.5) node[] {\tiny$Q$};
    \draw (.5,1.3) node[] {\tiny$Q$};
    \draw (.5,-.3) node[] {\tiny$R$};
    \foreach \off in {0,1,2,3,4}
      \draw (1.5+\off*2,.5) node[] {\tiny$QR$};
    \foreach \off in {0,1,2,3}
      \draw[Red,ultra thick] (2+\off*2,.5) -- (3+\off*2,.5);
    \draw[Red,ultra thick,rounded corners=6] (10,.5) -- (10.5,.5) -- (10.5,1);  
    \draw[ultra thick,rounded corners=6] (10.5,0) -- (10.5,.5) -- (11,.5);  
    \draw[Purple,ultra thick,rounded corners=6] (10,.5) -- (10.2,.3);  
    \draw (10.5,1.3) node[] {\tiny$R$};
    \draw (10.5,-.3) node[] {\tiny$0$};
    \draw (11.3,.5) node[] {\tiny$0$};
    \draw (10.2,.17) node[] {\tiny$Q$};
    \foreach \off in {0}
      {
      \draw[ultra thick] (2.5+\off*10,0) -- (2.5+\off*10,1);
      \draw (2.5+\off*10,-.3) node[] {\tiny$0$};
      \draw (2.5+\off*10,1.3) node[] {\tiny$0$};
      \draw[Blue,ultra thick] (4.5+\off*10,0) -- (4.5+\off*10,1);
      \draw (4.5+\off*10,-.3) node[] {\tiny$1$};
      \draw (4.5+\off*10,1.3) node[] {\tiny$1$};
      \draw[Red,ultra thick] (6.5+\off*10,0) -- (6.5+\off*10,1);
      \draw (6.5+\off*10,-.3) node[] {\tiny$R$};
      \draw (6.5+\off*10,1.3) node[] {\tiny$R$};
      \draw[Purple,ultra thick] (8.5+\off*10,0) -- (8.5+\off*10,1);
      \draw (8.5+\off*10,-.3) node[] {\tiny$Q$};
      \draw (8.5+\off*10,1.3) node[] {\tiny$Q$};
      }
  \end{scope}
  % 3rd row
  \foreach \xoffset in {0,1,2,3,4,5,6,7,8,9,10}
  {
  \begin{scope}[shift={(\xoffset*2+.5,2)}]
    \draw (0,0) -- (0,1) -- (1,1) -- (1,0) -- (0,0);
  \end{scope}
  }
  \begin{scope}[shift={(.5,2)}]
    \draw[Red,ultra thick,rounded corners=6] (0,.5) -- (.5,.5) -- (.5,1);  
    \draw[Purple,ultra thick,rounded corners=6] (.5,0) -- (.5,.5) -- (1,.5);  
    \draw[Red,ultra thick,rounded corners=6] (1,.5) -- (.5,.5) -- (.5,1);  
    \draw (-.3,.5) node[] {\tiny$R$};
    \draw (.5,1.3) node[] {\tiny$R$};
    \draw (.5,-.3) node[] {\tiny$Q$};
    \foreach \off in {0,1,2,3,4}
      \draw (1.5+\off*2,.5) node[] {\tiny$RQ$};
    \foreach \off in {0,1,2,3}
      \draw[Purple,ultra thick] (2+\off*2,.5) -- (3+\off*2,.5);
    \draw[Purple,ultra thick,rounded corners=6] (10,.5) -- (10.5,.5) -- (10.5,1);  
    \draw[Blue,ultra thick,rounded corners=6] (10.5,0) -- (10.5,.5) -- (11,.5);  
    \draw[Purple,ultra thick,rounded corners=6] (11,.5) -- (10.5,.5) -- (10.5,1);  
    \draw (10.5,1.3) node[] {\tiny$Q$};
    \draw (10.5,-.3) node[] {\tiny$1$};
    \foreach \off in {0,1,2,3,4}
      \draw (11.5+\off*2,.5) node[] {\tiny$RQ1$};
    \foreach \off in {0,1,2,3}
      \draw[Blue,ultra thick] (12+\off*2,.5) -- (13+\off*2,.5);
    \draw[Blue,ultra thick,rounded corners=6] (20,.5) -- (20.5,.5) -- (20.5,1);
    \draw[ultra thick,rounded corners=6] (20.5,0) -- (20.5,.5) -- (21,.5);  
    \draw[Red,ultra thick,rounded corners=6] (20,.5) -- (20.2,.3);
    \draw (20.5,1.3) node[] {\tiny$1$};
    \draw (20.5,-.3) node[] {\tiny$0$};
    \draw (21.3,.5) node[] {\tiny$0$};
    \draw (20.2,.17) node[] {\tiny$RQ$};
    \foreach \off in {0,1}
      {
      \draw[ultra thick] (2.5+\off*10,0) -- (2.5+\off*10,1);
      \draw (2.5+\off*10,-.3) node[] {\tiny$0$};
      \draw (2.5+\off*10,1.3) node[] {\tiny$0$};
      \draw[Blue,ultra thick] (4.5+\off*10,0) -- (4.5+\off*10,1);
      \draw (4.5+\off*10,-.3) node[] {\tiny$1$};
      \draw (4.5+\off*10,1.3) node[] {\tiny$1$};
      \draw[Red,ultra thick] (6.5+\off*10,0) -- (6.5+\off*10,1);
      \draw (6.5+\off*10,-.3) node[] {\tiny$R$};
      \draw (6.5+\off*10,1.3) node[] {\tiny$R$};
      \draw[Purple,ultra thick] (8.5+\off*10,0) -- (8.5+\off*10,1);
      \draw (8.5+\off*10,-.3) node[] {\tiny$Q$};
      \draw (8.5+\off*10,1.3) node[] {\tiny$Q$};
      }
  \end{scope}
  % 4th row
  \foreach \xoffset in {0,1,2,3,4,5,6,7,8,9,10}
  {
  \begin{scope}[shift={(\xoffset*2+.5,-.5)}]
    \draw (0,0) -- (0,1) -- (1,1) -- (1,0) -- (0,0);
  \end{scope}
  }
  \begin{scope}[shift={(.5,-.5)}]
    \draw[Purple,ultra thick,rounded corners=6] (0,.5) -- (.5,.5) -- (.5,1);  
    \draw[Red,ultra thick,rounded corners=6] (.5,0) -- (.5,.5) -- (1,.5);  
    \draw[Purple,ultra thick,rounded corners=6] (1,.5) -- (.5,.5) -- (.5,1);  
    \draw (-.3,.5) node[] {\tiny$Q$};
    \draw (.5,1.3) node[] {\tiny$Q$};
    \draw (.5,-.3) node[] {\tiny$R$};
    \foreach \off in {0,1,2,3,4}
      \draw (1.5+\off*2,.5) node[] {\tiny$QR$};
    \foreach \off in {0,1,2,3}
      \draw[Red,ultra thick] (2+\off*2,.5) -- (3+\off*2,.5);
    \draw[Red,ultra thick,rounded corners=6] (10,.5) -- (10.5,.5) -- (10.5,1);  
    \draw[Blue,ultra thick,rounded corners=6] (10.5,0) -- (10.5,.5) -- (11,.5);  
    \draw[Red,ultra thick,rounded corners=6] (11,.5) -- (10.5,.5) -- (10.5,1);  
    \draw (10.5,1.3) node[] {\tiny$Q$};
    \draw (10.5,-.3) node[] {\tiny$1$};
    \foreach \off in {0,1,2,3,4}
      \draw (11.5+\off*2,.5) node[] {\tiny$QR1$};
    \foreach \off in {0,1,2,3}
      \draw[Blue,ultra thick] (12+\off*2,.5) -- (13+\off*2,.5);
    \draw[Blue,ultra thick,rounded corners=6] (20,.5) -- (20.5,.5) -- (20.5,1);
    \draw[ultra thick,rounded corners=6] (20.5,0) -- (20.5,.5) -- (21,.5);  
    \draw[Purple,ultra thick,rounded corners=6] (20,.5) -- (20.2,.3);
    \draw (20.5,1.3) node[] {\tiny$1$};
    \draw (20.5,-.3) node[] {\tiny$0$};
    \draw (21.3,.5) node[] {\tiny$0$};
    \draw (20.2,.17) node[] {\tiny$QR$};
    \foreach \off in {0,1}
      {
      \draw[ultra thick] (2.5+\off*10,0) -- (2.5+\off*10,1);
      \draw (2.5+\off*10,-.3) node[] {\tiny$0$};
      \draw (2.5+\off*10,1.3) node[] {\tiny$0$};
      \draw[Blue,ultra thick] (4.5+\off*10,0) -- (4.5+\off*10,1);
      \draw (4.5+\off*10,-.3) node[] {\tiny$1$};
      \draw (4.5+\off*10,1.3) node[] {\tiny$1$};
      \draw[Red,ultra thick] (6.5+\off*10,0) -- (6.5+\off*10,1);
      \draw (6.5+\off*10,-.3) node[] {\tiny$R$};
      \draw (6.5+\off*10,1.3) node[] {\tiny$R$};
      \draw[Purple,ultra thick] (8.5+\off*10,0) -- (8.5+\off*10,1);
      \draw (8.5+\off*10,-.3) node[] {\tiny$Q$};
      \draw (8.5+\off*10,1.3) node[] {\tiny$Q$};
      }
  \end{scope}
  % last graph
  \begin{scope}[shift={(2,-3)}]
    \draw (0,0) -- (0,1) -- (1,1) -- (1,0) -- (0,0);
  \end{scope}
  \foreach \xoffset in {0,1}
  {
  \begin{scope}[shift={(8+\xoffset*7.5,-3)}]
    \draw (0,0) -- (0,1) -- (1,1) -- (1,0) -- (0,0);
  \end{scope}
  }
  \foreach \xoffset in {0,1,2,3,4,5,6,7,8,9,10,11,12,13}
  {
  \begin{scope}[shift={(.5+\xoffset*1.5,-4.5)}]
    \draw (0,0) -- (0,1) -- (1,1) -- (1,0) -- (0,0);
  \end{scope}
  }
  % above/below
  \draw (2.5,-1.75) node[] {\tiny$0$};
  \foreach \off in {0,6,11}
    {
    \draw (1+\off*1.5,-3.25) node[] {\tiny$0$};
    \draw (1+\off*1.5,-4.75) node[] {\tiny$0$};
    }
  \foreach \off in {0,8}
    \draw (2.5+\off*1.5,-4.75) node[] {\tiny$0$};
  \foreach \off in {0,5}
    \draw (8.5+\off*1.5,-3.25) node[] {\tiny$0$};
  \draw (8.5,-1.75) node[] {\tiny$0$};
  \foreach \off in {0,5,10}
    {
    \draw (4+\off*1.5,-3.25) node[] {\tiny$1$};
    \draw (4+\off*1.5,-4.75) node[] {\tiny$1$};
    }
  \draw (2.5,-3.25) node[] {\tiny$1$};
  \foreach \off in {0,10}
    {
    \draw (5.5+\off*1.5,-3.25) node[] {\tiny$R$};
    \draw (5.5+\off*1.5,-4.75) node[] {\tiny$R$};
    }
  \draw (8.5,-4.75) node[] {\tiny$R$};
  \draw (14.5,-3.25) node[] {\tiny$R$};
  \foreach \off in {0,4}
    {
    \draw (7+\off*1.5,-3.25) node[] {\tiny$Q$};
    \draw (7+\off*1.5,-4.75) node[] {\tiny$Q$};
    }
  \foreach \off in {0,3}
    \draw (16,-4.75+\off) node[] {\tiny$Q$};
  % left/right
  \foreach \off in {0,4,5,10}
    \draw (1.75+\off*1.5,-2.5) node[] {\tiny$0$};
  \foreach \off in {0,1,2,3,8}
    \draw (3.25+\off*1.5,-4) node[] {\tiny$0$};
  \draw (3.25,-2.5) node[] {\tiny$1$};
  \foreach \off in {0,1}
    \draw (.25+\off*1.5,-4) node[] {\tiny$1$};
  \foreach \off in {0,1,2,3}
    \draw (9.25+\off*1.5,-4) node[] {\tiny$R$};
  \foreach \off in {0,1,2,3}
    \draw (16.75+\off*1.5,-4) node[] {\tiny$Q$};
  \draw (15.25,-2.5) node[] {\tiny$Q$};
  % pipes
  \foreach \off in {0,5}
    \draw[ultra thick,rounded corners=6] (8.5+\off*1.5,-3.5) -- (8.5+\off*1.5,-4) -- (8+\off*1.5,-4);
  \draw[Red,ultra thick,rounded corners=6] (8.5,-4.5) -- (8.5,-4) -- (9,-4);
  \foreach \off in {0,1,2}
    \draw[Red,ultra thick] (9.5+\off*1.5,-4) -- (10.5+\off*1.5,-4);
  \draw[Red,ultra thick,rounded corners=6] (14.5,-3.5) -- (14.5,-4) -- (14,-4);
  \draw[ultra thick,rounded corners=6] (14.5,-4.5) -- (14.5,-4) -- (15,-4);
  \draw[ultra thick,rounded corners=6] (8.5,-2) -- (8.5,-2.5) -- (8,-2.5);
  \draw[ultra thick,rounded corners=6] (8.5,-3) -- (8.5,-2.5) -- (9,-2.5);
  \draw[Purple,ultra thick,rounded corners=6] (16,-2) -- (16,-2.5) -- (15.5,-2.5);
  \draw[ultra thick,rounded corners=6] (16,-3) -- (16,-2.5) -- (16.5,-2.5);
  \draw[Purple,ultra thick,rounded corners=6] (16,-4.5) -- (16,-4) -- (16.5,-4);
  \draw[Blue,ultra thick] (1.5,-4) -- (.5,-4);
  \foreach \off in {0,1,2}
    \draw[Purple,ultra thick] (17+\off*1.5,-4) -- (18+\off*1.5,-4);
  \foreach \off in {0,1,2}
    \draw[ultra thick] (3.5+\off*1.5,-4) -- (4.5+\off*1.5,-4);
  \foreach \off in {0,9,16.5}
    \draw[ultra thick] (1+\off,-3.5) -- (1+\off,-4.5);
  \foreach \off in {0,7.5,15}
    \draw[Blue,ultra thick] (4+\off,-3.5) -- (4+\off,-4.5);
  \foreach \off in {0,15}
    \draw[Red,ultra thick] (5.5+\off,-3.5) -- (5.5+\off,-4.5);
  \foreach \off in {0,6}
    \draw[Purple,ultra thick] (7+\off,-3.5) -- (7+\off,-4.5);
  \draw[Blue,ultra thick,rounded corners=6] (3,-2.5) -- (2.5,-2.5) -- (2.5,-3);
  \draw[ultra thick,rounded corners=6] (3,-2.5-1.5) -- (2.5,-2.5-1.5) -- (2.5,-3-1.5);
  \draw[ultra thick,rounded corners=6] (2.5,-2) -- (2.5,-2.5) -- (2,-2.5);
  \draw[Blue,ultra thick,rounded corners=6] (2.5,-2-1*1.5) -- (2.5,-2.5-1*1.5) -- (2,-2.5-1*1.5);
  \end{tikzpicture}
\end{center}

The \defn{upper $K$-tiles} are the left-right (not up-down, any more!)
flipped versions of these tiles, with the roles of $0$ and $1$
exchanged.  In particular, no word $W$ properly contains $1$, and if
it contains $0$ then the $0$ must be at the end.
There are a total of 82 upper and lower $K$-tiles. 

Define a \defn{$K$-DS pipe dream} 
of partitions $\lambda \subseteq a \times b$ and $\mu \subseteq c \times d$ 
as a tiling of the region from the previous subsection, with the same 
restrictions on the tiles and the boundary labels as in the previous subsection, but using all types of lower and upper $K$-tiles, respectively. 

Also in this setting, the only letters appearing on the North labels
of the region are $0$ and $1$, which make for a bit string 
and hence partition $\nu(P)$.

Define the \defn{$S$-weight} of a $K$-DS pipe dream $P$ as 
\[
  \weight_S(P) := \prod_{\text{tiles } t }\begin{cases}
    1 - \exp(t) & \text{if $t$ is the equivariant tile (all $0$ labels)}  \\
    \exp(t) & \text{if $t$ is a fusor tile,} \\
    1 &  \text{otherwise}
  \end{cases}
\]
in $K_0^S(\pt) = \Z[\exp(\pm t)]$. (Actually it is natural to define the
weight of an equivariant tile in an illegal location to be $0$.)

Finally, define $\fusing(P) := |\nu| - |\lambda| - |\mu| + \#\{\text{equivariant tiles}\}$.
Equivalently, $\fusing(P)$ is the sum of the sizes of words $W$ appearing in fusor tiles of $P$. 
(So $\fusing(P) = 0$ iff the $K$-DS pipe dream $P$ is an ordinary DS pipe dream, 
since the presence of a displacer tile forces the appearance of a fusor tile with $|W| > 0$ to the East of it.)
We will say more about the two equivalent definitions of $\fusing(P)$ later.
%  i.e. it was a necessary concept in paper\#1 but now subsumed into
%  the formula quoted from there}

\begin{thm}
\label{thm:PSPipeDreams}
As elements of $K_0^S(\Gr{a+c}{a+b+d+c})$, the expansion 
  in the $K_0^S(\pt)$-basis is 
  \[
    [X^\lambda \oplus X^\mu] = \sum_{\text{$K$-DS-pipe dreams }P \text{ of partitions } \lambda \text{ and }\mu} 
    (-1)^{\fusing(P)} \weight_K(P) [X^{\nu(P)}].
  \]
\end{thm}

\begin{example}
  When expanding the equivariant $K$-class of $X^\lambOneOne \oplus X^\lambOneNaught$ in the basis of Schubert classes, 
  we find the DS pipe dreams from example \ref{example:H}, 
  and in addition, two DS pipe dreams containing $K$-tiles. 
  
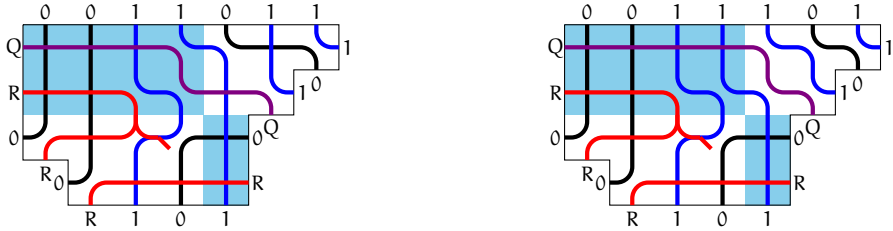
\begin{center}
  \begin{tikzpicture}[scale=.6]
  \foreach \xoffset in {0,12}
  {
    \begin{scope}[shift={(\xoffset,0)}]
  \fill[SkyBlue] (0,4) rectangle (4,2);
  \fill[SkyBlue] (4,2) rectangle (5,0);
  \draw (0,4) -- (0,1) -- (1,1) -- (1,0) -- (5,0) -- (5,2) -- (6,2) -- (6,3) -- (7,3) -- (7,4) -- (0,4);
  \foreach \x in {0,1}
    \draw (\x-.2,1.5-\x) node[] {\tiny$0$};
  \foreach \x in {0,2}
    \draw (\x+2.5,-.3) node[] {\tiny$1$};
  \foreach \x in {0,1}
    \draw (\x+6.2,2.5+\x) node[] {\tiny$1$};
  \foreach \x in {0,1}
    \draw (0.5 + \x,1-\x-.3) node[] {\tiny$R$};
  \draw (0-.2,3.5) node[] {\tiny$Q$};
  \draw (0-.2,2.5) node[] {\tiny$R$};
  \draw (5+.2,0.5) node[] {\tiny$R$};
  \draw (5+.2,1.5) node[] {\tiny$0$};
  \draw (3.5,0-.3) node[] {\tiny$0$};
  \draw (5.5,2-.3) node[] {\tiny$Q$};
  \draw (6.5,3-.3) node[] {\tiny$0$};
    \end{scope}
  }
  % 1st set of pipes
  \draw[ultra thick,rounded corners=6] (4.5,4) -- (4.5,3.5) -- (6.5,3.5) -- (6.5,3);
  \draw[ultra thick,rounded corners=6] (1.5,4) -- (1.5,.5) -- (1,.5);
  \draw[ultra thick,rounded corners=6] (.5,4) -- (.5,1.5) -- (0,1.5);
  \draw[ultra thick,rounded corners=6] (3.5,0) -- (3.5,1.5) -- (5,1.5);
  \draw[Blue,ultra thick,rounded corners=6] (6.5,4) -- (6.5,3.5) -- (7,3.5);
  \draw[Blue,ultra thick,rounded corners=6] (5.5,4) -- (5.5,2.5) -- (6,2.5);
  \draw[Blue,ultra thick,rounded corners=6] (3.5,4) -- (3.5,3.5) -- (4.5,3.5) -- (4.5,0);
  \draw[Blue,ultra thick,rounded corners=6] (2.5,4) -- (2.5,2.5) -- (3.5,2.5) -- (3.5,1.5) -- (2.5,1.5) -- (2.5,0);
  \draw[Purple,ultra thick,rounded corners=6] (0,3.5) -- (3.5,3.5) -- (3.5,2.5) -- (5.5,2.5) -- (5.5,2);
  \draw[Red,ultra thick,rounded corners=6] (0,2.5) -- (2.5,2.5) -- (2.5,1.5) -- (.5,1.5) -- (.5,1);
  \draw[Red,ultra thick,rounded corners=6] (2.5,2) -- (2.5,1.5) -- (3,1.5); 
  \draw[Red,ultra thick,rounded corners=6] (3,1.5) -- (3.25,1.25);
  \draw[Red,ultra thick,rounded corners=6] (1.5,0) -- (1.5,.5) -- (5,.5); 
  \foreach \x in {0,1,4}
    \draw (\x+.5,4.3) node[] {\tiny$0$};
  \foreach \x in {2,3,5,6}
    \draw (\x+.5,4.3) node[] {\tiny$1$};
  % 2nd set of pipes 
  \draw[ultra thick,rounded corners=6] (17.5,4) -- (17.5,3.5) -- (18.5,3.5) -- (18.5,3);
  \draw[ultra thick,rounded corners=6] (13.5,4) -- (13.5,.5) -- (13,.5);
  \draw[ultra thick,rounded corners=6] (.5+12,4) -- (.5+12,1.5) -- (0+12,1.5);
  \draw[ultra thick,rounded corners=6] (3.5+12,0) -- (3.5+12,1.5) -- (5+12,1.5);
  \draw[Blue,ultra thick,rounded corners=6] (6.5+12,4) -- (6.5+12,3.5) -- (7+12,3.5);
  \draw[Blue,ultra thick,rounded corners=6] (16.5,4) -- (16.5,3.5) -- (17.5,3.5) -- (17.5,2.5) -- (18,2.5);
  \draw[Blue,ultra thick,rounded corners=6] (15.5,4) -- (15.5,2.5) -- (16.5,2.5) -- (16.5,0);
  \draw[Blue,ultra thick,rounded corners=6] (14.5,4) -- (14.5,2.5) -- (15.5,2.5) -- (15.5,1.5) -- (14.5,1.5) -- (14.5,0);
  \draw[Purple,ultra thick,rounded corners=6] (12,3.5) -- (16.5,3.5) -- (16.5,2.5) -- (17.5,2.5) -- (17.5,2);
  \draw[Red,ultra thick,rounded corners=6] (12,2.5) -- (14.5,2.5) -- (14.5,1.5) -- (12.5,1.5) -- (12.5,1);
  \draw[Red,ultra thick,rounded corners=6] (14.5,2) -- (14.5,1.5) -- (15,1.5);
  \draw[Red,ultra thick,rounded corners=6] (15,1.5) -- (15.25,1.25);
  \draw[Red,ultra thick,rounded corners=6] (1.5+12,0) -- (1.5+12,.5) -- (5+12,.5); 
  \foreach \x in {0,1,5}
    \draw (\x+12.5,4.3) node[] {\tiny$0$};
  \foreach \x in {0,1,2,4}
    \draw (\x+14.5,4.3) node[] {\tiny$1$};
  \end{tikzpicture}
  \end{center}
  Both $K$-DS pipe dreams have their genuine $K$-tiles in the lower half, 
  displacer and fusor
  
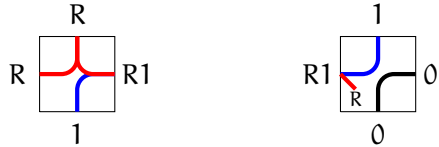
\begin{center}
  \begin{tikzpicture}[scale=1]
  \foreach \xoffset in {0,4}
  {
    \begin{scope}[shift={(\xoffset,0)}]
  \draw (0,0) -- (0,1) -- (1,1) -- (1,0) -- (0,0);
    \end{scope}
  }
  \draw[Blue,ultra thick,rounded corners=6] (1,.5) -- (.5,.5) -- (.5,0);
  \draw[Red,ultra thick,rounded corners=6] (.5,1) -- (.5,.5) -- (0,.5);
  \draw[Red,ultra thick,rounded corners=6] (.5,1) -- (.5,.5) -- (1,.5);
  \draw[Blue,ultra thick,rounded corners=6] (4.5,1) -- (4.5,.5) -- (4,.5);
  \draw[ultra thick,rounded corners=6] (5,.5) -- (4.5,.5) -- (4.5,0);
  \draw[Red,ultra thick] (4,.5) -- (4.2,.3);
  \draw (-.3,.5) node[] {\small$R$};
  \draw (-.3+1.6,.5) node[] {\small$R1$};
  \draw (.5,-.3) node[] {\small$1$};
  \draw (.5,-.3+1.6) node[] {\small$R$};
  \draw (4-.3,.5) node[] {\small$R1$};
  \draw (4-.3+1.5,.5) node[] {\small$0$};
  \draw (4.5,-.3) node[] {\small$0$};
  \draw (4.5,-.3+1.6) node[] {\small$1$};
  \draw (4.2,.17) node[] {\tiny$R$};
  \end{tikzpicture}
  \end{center}
  at $(2,2)$ and $(3,2)$, with words $W = \emptyset$ and $W = R$, respectively. 
  The second $K$-DS pipe dream contains an equivariant 1-1 elbow at $(4,2)$. 
  Therefore, 
  \begin{eqnarray*}
    \left[ X^\lambOneOne \oplus X^\lambOneNaught \right]
    & = & \left[ X^\lambTwoOneNaughtNaught \right]
    + \left[ X^\lambOneOneOneNaught \right] 
    + (1 - \exp(t)) \left[ X^\lambTwoOneOneNaught \right] 
     -  \exp(t)\left[ X^\lambOneOneNaughtNaught \right] 
    - (1 - \exp(t))\exp(t) \left[ X^\lambOneOneOneNaught \right] \\
    & = & \left[ X^\lambTwoOneNaughtNaught \right]
    + (1 - \exp(t) + \exp(2t) )\left[ X^\lambOneOneOneNaught \right] 
     +  (1 - \exp(t)) \left[ X^\lambTwoOneOneNaught \right] 
    - \exp(t) \left[ X^\lambOneOneNaughtNaught \right] .
  \end{eqnarray*}
  as elements of $K_0^S(\Gr{4}{7})$. 

  When expanding the equivariant $K$-class of $X^\lambOneOne \oplus X^\lambOneNaught$ in the basis of Schubert classes, 
  we find the three DS pipe dreams from example \ref{example:H}, 
  and in addition, two DS pipe dreams
  
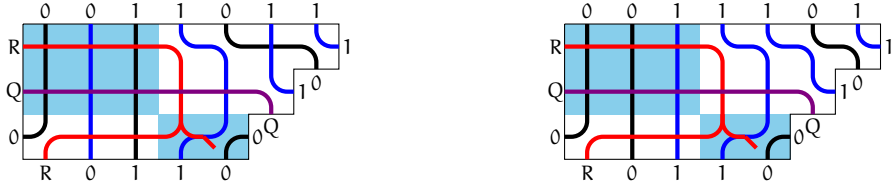
\begin{center}
  \begin{tikzpicture}[scale=.6]
  \foreach \xoffset in {0,12}
  {
    \begin{scope}[shift={(\xoffset,0)}]
  \fill[SkyBlue] (0,3) rectangle (3,1);
  \fill[SkyBlue] (3,1) rectangle (5,0);
  \draw (0,3) -- (0,0) -- (5,0) -- (5,1) -- (6,1) -- (6,2) -- (7,2) -- (7,3) -- (0,3);
  \draw (-.2,.5) node[] {\tiny$0$};
  \foreach \x in {0,1}
    \draw (\x+2.5,-.3) node[] {\tiny$1$};
  \foreach \x in {0,1}
    \draw (\x+6.2,1.5+\x) node[] {\tiny$1$};
  \draw (0.5,-.3) node[] {\tiny$R$};
  \draw (0-.2,1.5) node[] {\tiny$Q$};
  \draw (0-.2,2.5) node[] {\tiny$R$};
  \draw (5+.2,0.5) node[] {\tiny$0$};
  \foreach \x in {0,3}
  \draw (1.5 + \x,-.3) node[] {\tiny$0$};
  \draw (5.5,1-.3) node[] {\tiny$Q$};
  \draw (6.5,2-.3) node[] {\tiny$0$};
  \end{scope}
  }
  % 1st set of pipes
  \draw[ultra thick,rounded corners=6] (.5,3) -- (.5,.5) -- (0,.5);
  \draw[ultra thick,rounded corners=6] (2.5,3) -- (2.5,0);
  \draw[ultra thick,rounded corners=6] (4.5,3) -- (4.5,2.5) -- (6.5,2.5) -- (6.5,2); 
  \draw[ultra thick,rounded corners=6] (5,.5) -- (4.5,.5) -- (4.5,0);
  \draw[Blue,ultra thick,rounded corners=6] (1.5,3) -- (1.5,0);
  \draw[Blue,ultra thick,rounded corners=6] (3.5,3) -- (3.5,2.5) -- (4.5,2.5) -- (4.5,.5) -- (3.5,.5) -- (3.5,0);
  \draw[Blue,ultra thick,rounded corners=6] (5.5,3) -- (5.5,1.5) -- (6,1.5); 
  \draw[Blue,ultra thick,rounded corners=6] (6.5,3) -- (6.5,2.5) -- (7,2.5); 
  \draw[Red,ultra thick,rounded corners=6] (0,2.5) -- (3.5,2.5) -- (3.5,.5) -- (.5,.5) -- (.5,0);
  \draw[Red,ultra thick,rounded corners=6] (3.5,1) -- (3.5,.5) -- (4,.5); 
  \draw[Red,ultra thick,rounded corners=6] (4,.5) -- (4.25,.25);
  \draw[Purple,ultra thick,rounded corners=6] (0,1.5) -- (5.5,1.5) -- (5.5,1);
  \foreach \x in {0,1,4}
    \draw (\x+.5,3.3) node[] {\tiny$0$};
  \foreach \x in {2,3,5,6}
    \draw (\x+.5,3.3) node[] {\tiny$1$};
  % 2nd set of pipes
  \draw[ultra thick,rounded corners=6] (12.5,3) -- (12.5,.5) -- (12,.5);
  \draw[ultra thick,rounded corners=6] (13.5,3) -- (13.5,0);
  \draw[ultra thick,rounded corners=6] (17.5,3) -- (17.5,2.5) -- (18.5,2.5) -- (18.5,2);
  \draw[ultra thick,rounded corners=6] (17,.5) -- (16.5,.5) -- (16.5,0);
  \draw[Blue,ultra thick,rounded corners=6] (14.5,3) -- (14.5,0);
  \draw[Blue,ultra thick,rounded corners=6] (15.5,3) -- (15.5,2.5) -- (16.5,2.5) -- (16.5,.5) -- (15.5,.5) -- (15.5,0);
  \draw[Blue,ultra thick,rounded corners=6] (16.5,3) -- (16.5,2.5) -- (17.5,2.5) -- (17.5,1.5) -- (18,1.5);
  \draw[Blue,ultra thick,rounded corners=6] (18.5,3) -- (18.5,2.5) -- (19,2.5);
  \draw[Red,ultra thick,rounded corners=6] (12,2.5) -- (15.5,2.5) -- (15.5,.5) -- (12.5,.5) -- (12.5,0);
  \draw[Red,ultra thick,rounded corners=6] (15.5,1) -- (15.5,.5) -- (16,.5); 
  \draw[Red,ultra thick,rounded corners=6] (16,.5) -- (16.25,.25);
  \draw[Purple,ultra thick,rounded corners=6] (12,1.5) -- (17.5,1.5) -- (17.5,1);
  \foreach \x in {0,1,5}
    \draw (\x+12.5,3.3) node[] {\tiny$0$};
  \foreach \x in {2,3,4,6}
    \draw (\x+12.5,3.3) node[] {\tiny$1$};
  \end{tikzpicture}
  \end{center}
  having the same genuine $K$-tiles and equivariant tiles. 
\end{example}

\subsection{Outline of the paper}

\junk{
While we have attempted to make the statements in this paper self-contained, 
the proofs rely heavily on concepts developed in \cite{K} and we admit
that this paper would be difficult to read without familiarity with \cite{K}.
On the other hand, the only new geometry in this paper (beside invocation
of results from \cite{K}) is the use of direct sum and Grassmannian duality,
the rest being combinatorics of pipe dreams and associated objects.
}

In \S \ref{ssec:positroidvars} we recall the definitions of and
around positroid varieties, and compute how Grassmannian duality acts on them.
In \S \ref{ssec:KTformula} we recall the expansion formula 
from \cite{K} that we will twice have to use.
The real work is in \S \ref{sec:degen}, 
where we set up the direct sum variety in \S \ref{ssec:tostart}, 
apply the formula the first time in \S \ref{ssec:bottombdcrows},
use Grassmannian duality in \S \ref{ssec:midsort},
and apply the formula the second time in \S \ref{ssec:finalrows},
with summary in \S \ref{ssec:DSProof}.

\section{Interval positroid varieties and IP pipe dreams}

\subsection{Positroid varieties, interval positroid varieties, and duality}
\label{ssec:positroidvars}

Our reference for this section is \cite{KLS}.

Each of the expansion formul\ae\ relates a direct sum variety 
$X^\lambda \oplus X^\mu$ with Schubert varieties $\{X^\nu\}$.
To prove them, we make use of a class of varieties that 
interpolates between them, {\em interval positroid varieties}
(studied in \cite{K} and implicitly, in \cite{BC}).
While \cite{K} gives expansion formul\ae\ for classes of general interval
positroid varieties, its equivariant formulae are not positive in the senses
required in this paper, as will be explained in \S \ref{ssec:bottombdcrows}.

First define a \defn{bounded juggling pattern $J :\integers\to \integers$}
as a bijection satisfying $J(i+n) = J(i)+n$ for all $i\in \integers$
(making it an \defn{affine permutation}), with each $J(i)-i \in [0,n]$.
The average value of $J(i)-i$ is necessarily an integer $k\in [0,n]$,
called the \defn{ball number}. 
See Figure \ref{fig:goodVariety} for an example; 
the picture only shows the union of all $\{i\} \times (-i+[0,n])$ in $\integers\times\integers$, 
which suffices to characterize $J$. 

\begin{center}
\begin{figure}[ht]
    \begin{tikzpicture}[scale=.25]
    \foreach \offset in {0,1,2,3,4,5,6,7,8,9,10,11,12,13,14,15,16,17,18,19,20,21,22,23,24,25}
    {
      \begin{scope}[shift={(\offset,-\offset)}]
        \draw (0,0) -- (0,-1) -- (1,-1);
        \draw (12,0) -- (12,-1) -- (13,-1);
      \end{scope}
    }
    \foreach \offset in {0,11}
    {
      \begin{scope}[shift={(\offset,-\offset)}]
        \fill (3.5,-2.5) circle (0.25);
        \fill (8.5,-7.5) circle (0.25);
        \fill (9.5,-8.5) circle (0.25);
        \fill (10.5,-5.5) circle (0.25);
        \fill (11.5,-6.5) circle (0.25);
        \fill (12.5,-10.5) circle (0.25);
        \fill (13.5,-3.5) circle (0.25);
        \fill (15.5,-4.5) circle (0.25);
        \fill (16.5,-9.5) circle (0.25);
        \fill (17.5,-11.5) circle (0.25);
        \fill (18.5,-12.5) circle (0.25);
      \end{scope}
    }
    \fill (6.5,-.5) circle (0.25);
    \fill (7.5,-1.5) circle (0.25);
    \fill (25.5,-24.5) circle (0.25);
    \fill (35.5,-25.5) circle (0.25);
    \draw[Red] (2,0) -- (2,-2) -- (13,-2) -- (13,-13) -- (24,-13) -- (24,-24) -- (35,-24) -- (35,-26);
    \end{tikzpicture}
  \caption{A bounded juggling pattern}
  \label{fig:goodVariety}
\end{figure}
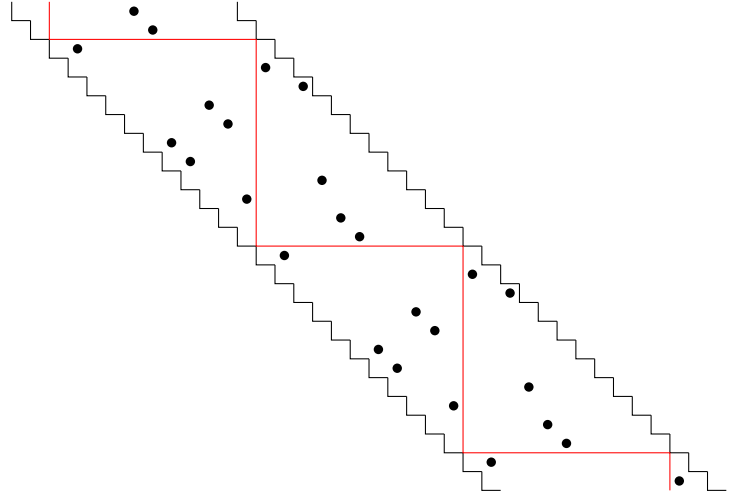
\end{center}

Given a matrix $M \in M_{k \times n}$, let $(\vec v_i)_{i\in \integers}$ be the 
infinite periodic list of column vectors, 
$\vec v_i :=$ column $i\bmod n$ of $M$,
and associate a bounded juggling pattern $J_M$ defined by
$$ J_M(i) := \min\ \{j\geq i \ 
:\ \vec v_i \in span(\vec v_{i+1},\ldots,\vec v_j) \} $$
The boundedness and periodicity properties of $J_M$ are obvious, the bijectivity
less so but not difficult. 

For $J$ a bounded juggling pattern of period $n$ and ball number $k$, let
$$ \Pi_J 
:= \overline{ \left\{rowspan(M) \ :\ M \in M_{k\times n},\ J_M = J \right\} }
\quad\subseteq \Gr{k}{n} $$
be the associated \defn{positroid variety}. Scheme-theoretically, 
it is given by the rank conditions
$$ rank\left(M_{[i,j]} := \text{columns $[i,j]\bmod n$ in $M$}\right) 
\quad\leq\quad \big|\ [i,j] \setminus J([i,j])\ \big|,\qquad \forall i\leq j.
$$
Some of the rank conditions defining a positroid variety
$\Pi_J$ follow from others.  A box $(i, j)$ is called
\defn{essential} for $j$ if its rank condition 
${\rm rank}(M_{[i,j]}) \leq r_{i,j}$ is not implied by the rank
condition for any of $(i \pm 1, j)$, $(i, j \pm 1)$.  For defining
$\Pi_J$, it suffices to only impose essential rank conditions ${\rm
  rank}(M_{[i,j]}) \leq r_{i,j}$.

The periodicity condition on $J$ lets us reconstruct it from $J(1),\ldots,J(n)$,
or equivalently from the parallelogram with top $(1,1)-(1,n+1)$ 
and bottom $(n,n)-(n,2n)$ in the $\infty\times\infty$ permutation matrix.
Cut this parallelogram in half along a vertical line into 
$J$'s \defn{West triangle} and \defn{East triangle}. 
In the example of Figure \ref{fig:goodVariety}, 
the two triangles are delimited by red lines. 
By \cite[\S 2.2]{K}, the positroid variety $\Pi_J$ is 
\begin{itemize}
\item an \defn{interval positroid variety,} so called because
  the only essential conditions are on honest intervals $[i,j] \subseteq [1,n]$
  (not cyclic intervals), 
  iff the $k$ dots in the East triangle are Northwest to Southeast, 
\item a \defn{Richardson variety} iff in addition, the $n-k$ dots in the
  West triangle are Northwest to Southeast,
\item a \defn{Schubert variety}\footnote{We remind the reader that 
    the convention used here is opposite the cohomological one
    used in \cite{K}.}
  iff it is a Richardson variety and in addition, 
  the $n-k$ dots in the West triangle are flush North.
\end{itemize}

Most of the varieties considered in this paper are actually
{\em interval} positroid varieties, which by \cite[proposition 2.1]{K}
can be specified by giving just their West triangles,
upper triangular partial permutation matrices of rank $n-k$.
The juggling pattern from Figure \ref{fig:goodVariety} defines an interval positroid variety, 
provided the row indices in the West triangle below the red lines are $1,\ldots,n$. 

This subclass contains the class (depending also on a parameter $i\in [0,n]$) 
that will actually be of central interest in the paper, 
and we will call \emph{$i$-sorted}.
If the dots below row $i$ in the West triangle are in the
top rows $i+1,i+2,\ldots$ and run NW/SE, call this being \defn{$i$-sorted}.
Note that we sort from the back, so $i$-sorted implies $j$-sorted
when $i\leq j$, not $i\geq j$.
Call the West triangle or the bounded juggling pattern
it minimally extends to \defn{$i$-sorted}. 

\junk{
  The essential boxes appearing the juggling pattern of interval positroid 
  varieties are found at $(i,j)$ such that $1 \leq i \leq j \leq n$.  
  So the rank conditions necessary to define interval positroid
  varieties only involve honest intervals.  The essential boxes
  appearing the juggling pattern of Schubert varieties take the shape
  $(1,j)$; those appearing in opposite Schubert varieties the shape
  $(i,n)$; each Richardson variety is the intersection of a Schubert and
  an opposite Schubert variety.
}

In \cite{K}, the class of an arbitrary
interval positroid variety in $T$-equivariant $K$-homology is
expressed as a linear combination of Schubert classes. 
However, due to some reindexing involved in our choice of circle $S\leq T$,
that expansion is not positive in the sense required in this paper,
so our use of it will be slightly indirect.

\begin{Proposition}\label{prop:duality}
  Use the standard bilinear form on $\complexes^n$ to define a 
  ``duality'' isomorphism $D : \Gr{k}{n} \ \widetilde\to\ \Gr{n-k}{n}$. 
  Let $J : \integers\to \integers$ be a bounded juggling pattern. Then
  $$ D(\Pi_J) = \Pi_{J^{-1} \circ (i\mapsto i+n)}. $$
  This exchanges the West and East triangles of $J$, rotating $180^\circ$.
\end{Proposition}

\begin{proof}
  Each $\Pi_J$ is the intersection of ``meet-irreducible'' positroid varieties 
  given by single (essential) rank conditions
  $$ \{ M \in M_{k\times n}\ :\ rank(M_{[i,j]}) \leq |[i,j]|-
  \#\{\text{dots SW of $(i,j)$ in $J$} \}. $$
  This says that dots in $J$ come in four blocks of identity matrices: 
  \begin{itemize}
    \item Dots in the triangle SW of $(i,j)$ come in an identity block $I_1$
    sitting flush North and flush East in that triangle. 
    \item Dots in the remainder of the West triangle $[i,i+n]$ come in an identity block $I_2$
    sitting due South of $I_1$ and flush East in that triangle. 
    \item Dots in the complementary East triangle come in the unique NW-SE shape, 
    thus in identity blocks $I_3$ and $I_4$. 
  \end{itemize}
  
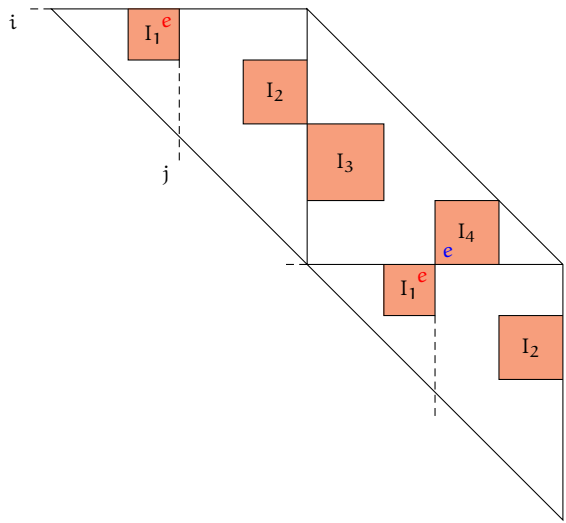
\begin{center}
    \begin{tikzpicture}[scale=.34]
      \foreach \x in {0,10}
        {
        \fill[Melon] (3+\x,20-\x) -- (3+\x,18-\x) -- (5+\x,18-\x) -- (5+\x,20-\x) -- (3+\x,20-\x);
        \fill[Melon] (10+\x,18-\x) -- (7.5+\x,18-\x) -- (7.5+\x,15.5-\x) -- (10+\x,15.5-\x) -- (10+\x,18-\x);
        \draw (3+\x,20-\x) -- (3+\x,18-\x) -- (5+\x,18-\x) -- (5+\x,20-\x);
        \draw (10+\x,18-\x) -- (7.5+\x,18-\x) -- (7.5+\x,15.5-\x) -- (10+\x,15.5-\x);
        \draw (4+\x,19-\x) node[] {\tiny $I_1$};
        \draw (8.75+\x,16.75-\x) node[] {\tiny $I_2$};
        \draw (4.5+\x,19.5-\x) node[Red] {\tiny $e$};
        \draw[densely dashed] (\x,20-\x) -- (-1+\x,20-\x);
        \draw[densely dashed] (5+\x,18-\x) -- (5+\x,14-\x);
        }
      \draw (-1.5,19.5) node[] {\tiny $i$};
      \draw (4.5,13.5) node[] {\tiny $j$};
      \fill[Melon] (10,15.5) -- (13,15.5)  -- (13,12.5) -- (10,12.5) -- (10,15.5);
      \fill[Melon] (15,10) -- (15,12.5) -- (17.5,12.5) -- (17.5,10) -- (15,10);
      \draw (10,15.5) -- (13,15.5)  -- (13,12.5) -- (10,12.5);
      \draw (15,10) -- (15,12.5) -- (17.5,12.5) -- (17.5,10);
      \draw (11.5,14) node[] {\tiny $I_3$};
      \draw (16.25,11.25) node[] {\tiny $I_4$};
      \draw (15.5,10.5) node[Blue] {\tiny $e$};
      \draw (10,20) -- (0,20) -- (20,0) -- (20,10) -- (10,20) -- (10,10) -- (20,10);
    \end{tikzpicture}
  \end{center}
  The essential condition defining $\Pi_J$ is indicated by an ${\color{Red}e}$.
  The dual of $\Pi_J$ 
  is also defined by just one rank condition, 
  $$ \{ M \in M_{n-k\times n}\ :\ rank(M_{[j+1,i-1]} \leq n-k-rank(I_1) \}. $$
  (Perhaps surprisingly, the size of $[i,j]$ does not show up on the right-hand side.)
  However, we want the rank bound to be $(n-|[i,j]|) - \#\{\text{dots strictly NE of $(i,j)$}\}$.
  
  Four equalities on the respective sizes of the blocks of $J$ are easy to check,
  \begin{eqnarray*}
  rank(I_1) + rank(I_2) & = & n-k , \\
  rank(I_3) + rank(I_4) & = & k , \\
  rank(I_1) + rank(I_3) & = & [i,j] , \\
  rank(I_2) + rank(I_4) & = & n-[i,j] .
  \end{eqnarray*}
  By induction over $rank(I_1)$ 
  (or equivalently, by induction over $J$ in the affine Bruhat order) 
  they imply a fifth equality, 
  \[
    rank(I_4) = rank(I_1) + k - |[i,j]| 
  \]
  giving the desired rank bound. 
  The corresponding essential condition is indicated by the letter ${\color{Blue}e}$ in the picture. 
%  (Alternatively, 
%  If $J(m) = m+k\ \forall m$, this is easy to check.
%  Then one confirms that if $J' \gtrdot J$ in affine Bruhat order,
%  the condition still holds. Induction on affine Bruhat order does the rest.
%
\junk{
  We count the dots in different regions of $J$: \\
  \centerline{\epsfig{file=duality2.eps,height=2in}}
  \begin{eqnarray*}
    a+b+d+c =& i-1 &= A+D+G+J \\
    E+F+G+H =& |[i,j]| &= H+K+B+E \\
    I+J+K+L =& n-j &= C+F+I+L    
  \end{eqnarray*}
  The blue region with $E$ dots in it is Southwest of position $[i,j]$,
  and the pink region (the one relevant in the dual) has
  $C+D+L$ dots, which we want to be $(n-|[i,j]|)-E$.
  \begin{eqnarray*}
    n-|[i,j]| 
    &=& (n-j) + (i-1) \\
    &=& 
  \end{eqnarray*}
}
\end{proof}

\subsection{The $K^T$-formula from \cite[\S 4.4]{K}}
\label{ssec:KTformula}

We recall several definitions from \cite{K}, to state (almost) its 
most precise formula, which involves cutting a triangle
$\{ (a \leq b) \} \subseteq [n]\times [n]$ into the squares
lexicographically before and after a position $(i,j)$. 
We simplify very slightly in the translation, in that we only need to
cut into the \defn{top half} consisting of rows $[1,i]$ and
\defn{bottom half} of rows $[i+1,n]$.

\subsubsection{Slices and their bounded juggling patterns}

The basic combinatorial objects in this paper (not the DS pipe dreams,
which are visibly composite) are
\begin{enumerate}
\item {\em $[h,j]$-partial pipe dreams}, which for each $i\in [h,j]$ give
\item {\em $i$-slices}, which give
\item $i$-sorted upper triangular partial permutations, which give
\item bounded juggling patterns with East triangle running NW/SE.
\end{enumerate}
We first address (2)$\to$(3)$\to$(4).
We have already discussed the last map; \cite[proposition 2.1]{K}
says that there is a unique such extension.

Define an \defn{$i$-slice} as a labeling with $0,1,R,Q$
of the following edges of the top half:
\begin{itemize}
\item the $n+1-i$ horizontal edges below row $i$ (above row $i+1$),
\item the $i$ vertical edges on the East side of rows $[1,i]$, and
\item the $i-1$ horizontal edges below the diagonal blocks.
\end{itemize}

(Notice that none of these edges are vertical and interior, which will
be why we don't have to consider the more complicated multiple labels 
that $K$-pieces can have.)

To an $i$-slice $s$, we attempt to associate an upper triangular
partial permutation matrix $\pi(s)$,
and if successful we call $s$ \defn{viable}. 

Draw rays perpendicular to the $L$-edges (for $L \in \{R,Q\}$), 
into the top half. There
should be the same number $m$ of vertical as horizontal rays -- or
else $s$ is not viable.  They should meet in $m^2$ locations inside
the top half, or else $s$ is not viable. Put dots along the diagonal
$m$ of these intersections.

For $1$-edges, much the same occurs, except there may be $m'$
horizontal $1$s and only $m < m'$ vertical $1$s. In this case ignore
the left $m'-m$ horizontal $1$s and only draw $m$ dots.
This completes the filling of the top half.

If there are $p$ $0$-labels on the South edge of $s$, 
put dots NW/SE in rows $[i+1,i+p]$ below those $0$s. 
Notice that the first dot will not land inside the bottom half
if the label under the $(i,i)$ square is $0$; this is the
last way to be non-viable. {\bf Hereafter all slices are
  assumed to be viable unless stated otherwise.}

The result is richer than an upper triangular partial permutation: its dots
come in $R$-type and $Q$-type, each group of which separately runs NW/SE.
We leave the reader to convince herself that

\begin{Lemma}\label{lem:slicetoperm}
  This constuction gives a bijection between $i$-slices, and
  $i$-sorted upper triangular partial permutations in which one has
  chosen a decomposition of the upper dots into the $R$-dots and $Q$-dots,
  each group of which separately runs NW/SE.
\end{Lemma}

Now that we have filled the triangle with corners $(1,1),(1,n),(n,n)$
with a partial permutation of dots, there is a unique way to extend it
to the permutation matrix of a bounded juggling pattern
that is NW/SE in its East triangle \cite[proposition 2.1]{K}.
This we declare to be $g(s)$. 

The most important special case is a $0$-slice. 
Since it has no vertical edges, it can have no letters at all,
only $0$s and $1$s. Then (as explained in \S \ref{ssec:positroidvars},
based on \cite[\S 2.2]{K}) the resulting $\Pi_{g(s)}$ is just a
Schubert\footnote{in the convention opposite to that of \cite{K}} variety.
At the other extreme, the $\Pi_{g(s)}$ for $n$-slices $s$ are 
Richardson varieties. (With more letters made available, as in \cite{K}, 
they are arbitrary interval positroid varieties.)

\subsubsection{Partial pipe dreams}

Define an \defn{$(i,j)$-partial pipe dream} $P$ (for $i\leq j$)
to be a viable $i$-slice and $j$-slice,
together with a filling with lower $K$-tiles of the trapezoid having corners 
$(i+1,i+1),(i+1,n),(j,n), (j,j)$, 
compatibly in the sense that any edge is labeled at most once.  
Note that $P$ has an associated $k$-slice $s_k(P)$ for each $k \in [i,j]$,
but instead of writing $g(s_k(P))$ we'll just write $g_k(P)$.

\begin{Theorem}\label{thm:partialKIP}
  \cite[special case of \S 4.4]{K}
  Fix a $j$-slice $s$, hence an interval positroid variety $\Pi_{g(s)}$.
  Let $\calP$ be the set of $(i,j)$-partial pipe dreams $P$ with $s_j(P)=s$.
  Then as classes in $K^{T}(\Gr{k}{n})$, 
  $$ [\Pi_{g(s)}] = \sum_{P\in \calP} (-1)^{\fusing(P)} \weight_K(P)\ [\Pi_{g_i(P)}] $$
  \junk{  Apparently we {\em do} need $fusing()$.}
  where 
  \[
  \weight_K(P) := \prod_{\text{tiles } t }\begin{cases}
    1 - \exp(y_j-y_i) & \begin{array}{l} \text{if } t \text{ is the equivariant tile (it has all 0s) at }(i,j) \end{array} \\
    \exp(y_j-y_i) & \begin{array}{l} \text{if } t \text{ is the lower fusor tile (it appears in the lower half} \\
    \text{and has 0s on its South and East) at }(i,j) \end{array} \\
    1 & \begin{array}{l} \text{otherwise} \end{array}
  \end{cases}
  \]
  in $K_0^T(\pt) = \Z[\exp(\pm y_1),\ldots,\exp(\pm y_1)]$, 
  and $\fusing(P)$ is the sum of the sizes of words $W$ appearing in fusor tiles of $P$. 
\end{Theorem}

\junk{Reminder to AK: reupload \cite{K} with this further improved version
  of \S 4.4}

This inductive theorem was developed for the case $(i,j) = (0,n)$,
in which case (in homology) it gives a Littlewood-Richardson rule.

\junk{
\begin{Proposition}\label{prop:uniquebottom}
  Let $s$ be a viable $n$-slice with $m$ dots in $g(s)$'s West half,
  rows $i+1$ to $n$. If they are in rows $i+1$ to $i+m$, and run NW/SE,
  then there is a unique $i$-partial pipe dream $P$ with $s_n(P) = s$.
  In each column, $s_i(P)$'s and $s_n(P)$'s horizontal labels agree.  
\end{Proposition}

\begin{proof}
  We will start with the $n$-slice, fill in the tiles one by one 
  in the area below row $i$, and see that the filling is unique.

  The labels on the vertical edges of $s$ in rows $i+1\ldots n$ 
  are $m$ nonzero labels, then $n-i-m$ $0$ labels.

  First we consider the last $n-i-m$ rows.
  Look at the rightmost column of tiles. It begins with a label 
  on the bottom (necessarily nonzero) and all $0$s on the right.
  So we must put in all crossing tiles in that column. 
  The result is a similar triangle, of size one smaller, 
  again with all $0$s on the right. Consequently these
  last $n-i-m$ rows are completely full of crossing tiles.
  In particular $s_n(P) = s_{i+m}(P)$.

  Now look at rows $i+1 \ldots i+m$. 
  \comment{row by row from bottom?}
\end{proof}
}

\section{Degeneration in stages, of direct sum varieties}
\label{sec:degen}

As a way to wrap up the multi-step proof in this section, 
we provide a road map in \S \ref{ssec:DSProof};
the reader may want to look ahead to there on a first reading.

\subsection{The positroid variety to start with}\label{ssec:tostart}

\begin{Proposition}\label{prop:bigStartingMatrix}
  Let $J$ be the (upper triangular) square matrix from $(1,1)$ to
  $(b+a,b+a)$ in the affine permutation matrix whose associated
  positroid variety is $X^\lambda$, and similarly $K$ the $(d+c)\times(d+c)$
  one for $X^\mu$. Construct an affine permutation matrix $\sigma$ as follows:
  transpose $J$ to $w_0 J^T w_0$ (again upper triangular),
  direct sum $K$ with that (again upper triangular),
  extend that as the West half of an affine permutation matrix
  whose East half runs NW/SE, and (as the bounding
  squares indicate) rotate forward by $a$.
  
%  \comment{The first of the following pics is yours, the second is mine. I left yours here to illustrate the difference.
%  Yours is true as well (except for the offset), 
%  but I don't see why you would talk about that juggling pattern here when you go on to use the other one.}
%  
%  \centerline{\epsfig{file=directsum2.eps,width=6in}}
  
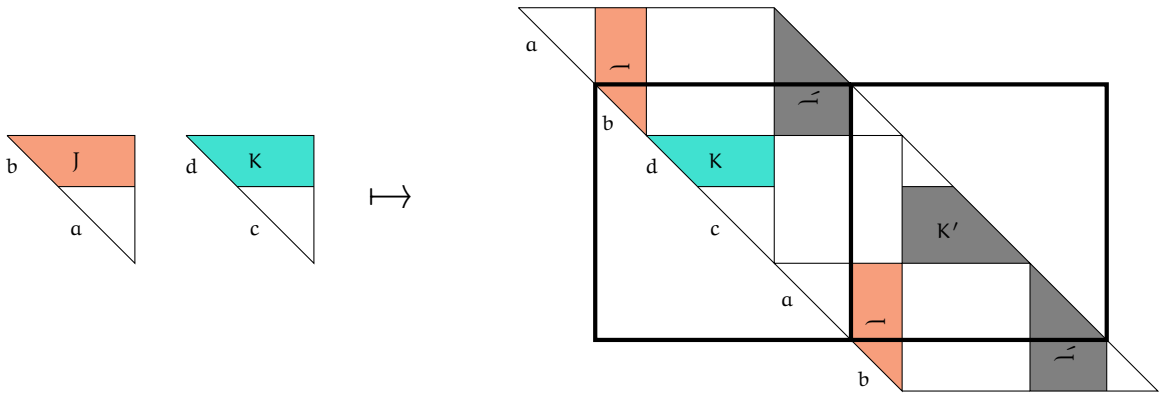
  \begin{center}
  \begin{tikzpicture}[scale=.34]
  \draw (15,2.5) node[] {\Large$\mapsto$};
  \foreach \offset in {(7,0),(25,0)}
  {
    \begin{scope}[shift={\offset}]
    \fill[Turquoise] (0,5) -- (2,3) -- (5,3) -- (5,5) -- (0,5);
    \draw (0,5) -- (5,0) -- (5,5) -- (0,5);
    \draw (2,3) -- (5,3);
    \draw (2.7,4) node[] {\tiny$K$};
    \draw (.2,3.8) node[] {\tiny$d$};
    \draw (2.7,1.3) node[] {\tiny$c$};
    \end{scope}
  }
  \fill[Melon] (0,5) -- (2,3) -- (5,3) -- (5,5) -- (0,5);
  \draw (2.7,4) node[] {\tiny$J$};
  \draw (.2,3.8) node[] {\tiny$b$};
  \draw (2.7,1.3) node[] {\tiny$a$};
  \draw (0,5) -- (5,0) -- (5,5) -- (0,5);
  \draw (2,3) -- (5,3);
  \foreach \offset in {(20,5),(30,-5)}
  {
    \begin{scope}[shift={\offset}]
    \fill[Melon] (3,2) -- (3,5) -- (5,5) -- (5,0) -- (3,2);
    \draw (0,5) -- (5,0) -- (5,5) -- (0,5);
    \draw (3,2) -- (3,5);
    \draw (4,2.7) node[rotate=-90] {\reflectbox{\tiny$J$}};
    \draw (.5,3.5) node[] {\tiny$a$};
    \draw (3.5,.5) node[] {\tiny$b$};
    \end{scope}
  }
  \foreach \offset in {(30,5),(40,-5)}
  {
    \begin{scope}[shift={\offset}]
    \fill[Gray] (3,2) -- (3,0) -- (0,0) -- (0,5) -- (3,2);
    \draw (0,5) -- (0,0) -- (5,0) -- (0,5);
    \draw (3,2) -- (3,0);
    \draw (1.4,1.5) node[rotate=-90] {\reflectbox{\tiny$J'$}};
    \end{scope}
  }
  \foreach \offset in {(35,0)}
  {
    \begin{scope}[shift={\offset}]
    \fill[Gray] (2,3) -- (0,3) -- (0,0) -- (5,0) -- (2,3);
    \draw (0,5) -- (0,0) -- (5,0) -- (0,5);
    \draw (2,3) -- (0,3);
    \draw (1.8,1.3) node[] {\tiny$K'$};
    \end{scope}
  }
  \draw (25,10) -- (30,10);
  \draw (35,-5) -- (40,-5);
  \draw[ultra thick] (33,7) -- (23,7) -- (23,-3) -- (33,-3) -- (33,7) -- (43,7) -- (43,-3) -- (33,-3);
  \end{tikzpicture}
  \end{center}

  Then $\Pi_\sigma$ (which is only a positroid variety, not interval
  positroid) is $K^S$-homologous to $X^\lambda \oplus X^\mu$.

  In each shaded region (namely $J,K,J',K'$) in this figure, the dots
  run NW/SE, and only these regions have dots.
\end{Proposition}

\begin{proof}%[Proof of Proposition \ref{prop:bigStartingMatrix}]
  First we consider the $\sigma$ pictured, then show how to arrive 
  at it geometrically.

  Since $J$ and $K$ have dots in all their top rows ($b$ and $d$ dots
  respectively), when we use $w_0 J^T w_0 \oplus K$ as the West half
  of an affine permutation matrix, we know that the East half will
  have no dots East of the North part of $K$ nor North of the 
  East part of $w_0 J^T w_0$. But also, the gray region $J'$
  East of the East part of $w_0 J^T w_0$ must have $a$ dots in it
  (since $w_0 J^T w_0$ misses that many rows out of $a+b$), 
  so has a dot in every column. Hence there are no dots South of $J'$.
  Similarly $K'$ (above $K$) has $c$ dots, and none West of it.
  We have shown that $\sigma$ has the form pictured.

  On the geometric side, Schubert varieties $X^\lambda$ and $X^\mu$ are defined by rank inequalities
  on points $\rowspan \left[ \begin{array}{c|c} B & A \end{array} \right] \in \Gr{a}{b+a}$ 
  and $\rowspan \left[ \begin{array}{c|c} D & C \end{array} \right] \in \Gr{c}{d+c}$ 
  where $B$, $A$, $D$ and $C$ are block matrices of sizes $b \times a$, $a \times a$, $d \times c$ and $c \times c$, respectively. 
  Since the weight of the $S$-action on $\Gr{c}{d+c}$ is 1 on block $B$ and 0 on block $A$, 
  flipping the numbering of columns within blocks $B$ and $A$ is an $S$-equivariant automorphism 
  \[
    \rowspan \left[ \begin{array}{c|c} B & A \end{array} \right] \mapsto 
    \rowspan \left[ \begin{array}{c|c} \overleftarrow{B} & \overleftarrow{A} \end{array} \right]
  \]
  of $\Gr{a}{b+a}$. We denote by $\overleftarrow{X}^\lambda$ the image of $X^\lambda$ under this automorphism. 
  Essential conditions on points in $X^\lambda$ only involve initial intervals columns starting in the first column. 
  Essential conditions on their flipped counterparts  
  therefore involve periodic intervals going from column $b$ to the left, 
  then wrapping around the matrix and going from the last column to the left. 
  
  We obtain the variety $\overleftarrow{X}^\lambda \oplus X^\mu$ 
  by imposing the rank inequalities defining $\overleftarrow{X}^\lambda$ and $X^\mu$, respectively, on the columns of matrix
  \[
    \left[ \begin{array}{c|c|c|c}
    \overleftarrow{B} & 0 & 0 & \overleftarrow{A} \\
    \hline
    0 & D & C & 0
    \end{array}\right] .
  \]
  Essential conditions on the blocks $\left[ \begin{array}{c|c} D & C \end{array}\right]$ of the matrix 
  are implemented by the block 
  
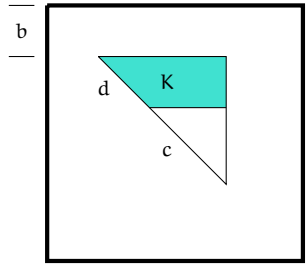
\begin{center}
  \begin{tikzpicture}[scale=.34]
    \fill[Turquoise] (0,5) -- (2,3) -- (5,3) -- (5,5) -- (0,5);
    \draw (0,5) -- (5,0) -- (5,5) -- (0,5);
    \draw (2,3) -- (5,3);
    \draw (2.7,4) node[] {\tiny$K$};
    \draw (.2,3.8) node[] {\tiny$d$};
    \draw (2.7,1.3) node[] {\tiny$c$};
    \draw [ultra thick] (-2,7) -- (-2,-3) -- (8,-3) -- (8,7) -- (-2,7);
    \draw (-2.5,7) -- (-3.5,7); 
    \draw (-2.5,5) -- (-3.5,5); 
    \draw (-3,6) node[] {\tiny$b$};
  \end{tikzpicture}
  \end{center}
  of $\sigma$. They only concern honest intervals ranging in columns $b+1,\ldots,b+d+c$ in the matrix. 
  Essential conditions on the complementary blocks of the matrix also involve periodic intervals. 
  They are implemented by the block 
  
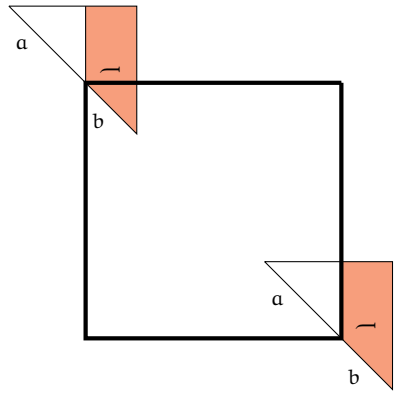
\begin{center}
  \begin{tikzpicture}[scale=.34]
  \foreach \offset in {(0,5),(10,-5)}
  {
    \begin{scope}[shift={\offset}]
    \fill[Melon] (3,2) -- (3,5) -- (5,5) -- (5,0) -- (3,2);
    \draw (0,5) -- (5,0) -- (5,5) -- (0,5);
    \draw (3,2) -- (3,5);
    \draw (4,2.5) node[rotate=-90] {\reflectbox{\tiny$J$}};
    \draw (.5,3.5) node[] {\tiny$a$};
    \draw (3.5,.5) node[] {\tiny$b$};
    \end{scope}
  }
  \draw[ultra thick] (13,7) -- (3,7) -- (3,-3) -- (13,-3) -- (13,7);
  \end{tikzpicture}
  \end{center}
  of $\sigma$. 
  Therefore, 
  \[
    \Pi_\sigma = \overleftarrow{X}^\lambda \oplus X^\mu 
  \]
  is $K^S$-homologous to $X^\lambda \oplus X^\mu$, and a positroid variety, not interval positroid. 
\end{proof}

Though it is not obvious why yet, 
it will turn out also to be useful to backward-rotate by $b$, thus working with matrices 
\[
  \left[ \begin{array}{c|c|c|c}
      \overleftarrow{A} & \overleftarrow{B} & 0 & 0 \\
  \hline
  0 & 0 & D & C
  \end{array}\right] .
\]
defining points in the Richardson variety associated to the affine permutation matrix
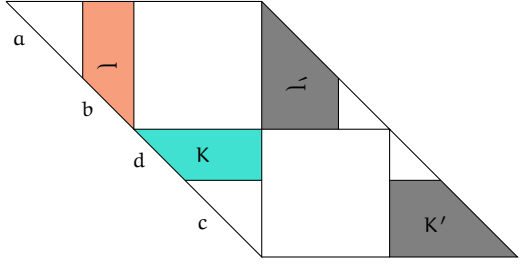
\begin{center}
\begin{tikzpicture}[scale=.34]
  \foreach \offset in {(25,0)}
  {
    \begin{scope}[shift={\offset}]
    \fill[Turquoise] (0,5) -- (2,3) -- (5,3) -- (5,5) -- (0,5);
    \draw (0,5) -- (5,0) -- (5,5) -- (0,5);
    \draw (2,3) -- (5,3);
    \draw (2.7,4) node[] {\tiny$K$};
    \draw (.2,3.8) node[] {\tiny$d$};
    \draw (2.7,1.3) node[] {\tiny$c$};
    \end{scope}
  }
  \foreach \offset in {(20,5)}
  {
    \begin{scope}[shift={\offset}]
    \fill[Melon] (3,2) -- (3,5) -- (5,5) -- (5,0) -- (3,2);
    \draw (0,5) -- (5,0) -- (5,5) -- (0,5);
    \draw (3,2) -- (3,5);
    \draw (4,2.5) node[rotate=-90] {\reflectbox{\tiny$J$}};
    \draw (.5,3.5) node[] {\tiny$a$};
    \draw (3.2,.8) node[] {\tiny$b$};
    \end{scope}
  }
  \foreach \offset in {(30,5)}
  {
    \begin{scope}[shift={\offset}]
    \fill[Gray] (3,2) -- (3,0) -- (0,0) -- (0,5) -- (3,2);
    \draw (0,5) -- (0,0) -- (5,0) -- (0,5);
    \draw (3,2) -- (3,0);
    \draw (1.4,1.8) node[rotate=-90] {\reflectbox{\tiny$J'$}};
    \end{scope}
  }
  \foreach \offset in {(35,0)}
  {
    \begin{scope}[shift={\offset}]
    \fill[Gray] (2,3) -- (0,3) -- (0,0) -- (5,0) -- (2,3);
    \draw (0,5) -- (0,0) -- (5,0) -- (0,5);
    \draw (2,3) -- (0,3);
    \draw (1.8,1.3) node[] {\tiny$K'$};
    \end{scope}
  }
  \draw (25,10) -- (30,10);
  \draw (30,0) -- (35,0);
\end{tikzpicture}
\end{center}
\noindent which we call $\sigma'$.
Of course this only gives the same class in $K^S(\Gr{a+c}{b+d+c+a})$ if we also
change the $S$-action, now to be weight $0$ on the first $a$
and last $c$ coordinates, and weight $1$ on the $b+d$ coordinates in the middle. 

\subsection{The bottom $b+d+c$ rows, of which $d+c$ are trivial}
\label{ssec:bottombdcrows}

Our goal now is to calculate the class 
$[\Pi_{\sigma'}] \in K^S(\Gr{a+c}{a+b+d+c})$,
where $S$ acts with weight $0$ on the first $b$
and last $d$ coordinates, and weight $1$ on the $a+c$ coordinates in the middle.
If the weights were monotonic, we could use theorem \ref{thm:partialKIP}
directly; since they instead go %$1^a 0^{b+d} 1^c$ 
$0^a 1^{b+d} 0^c$ it will take three steps, this being the first.
To emphasize: this nonmonotonicity is the reason that this paper is not a
straightforward corollary of \cite{K}.

\begin{Theorem}\label{thm:bpartial}
  Let $\sigma'$ be constructed from $J$ and $K$ as pictured above,
  and $S$ act on $\complexes^{a+b+d+c}$ with weight $1$ on the $b$ and
  $d$ groups of coordinates. 
  Then $\sigma'$ is $(a+b)$-sorted, 
  so there exists a unique $(a+b)$-slice $s$ with no $Q$-labels
  such that $g(s) = \sigma'$.
  Let $\calP := \{$the $(a+1,a+b)$-partial pipe dreams $P$ 
  with $s_{a+b}(P) = s\}$ (which will also have no $Q$s). Then
  $$ [\Pi_{\sigma'}] = \sum_{P \in \calP} \weight_S(P)\ [\Pi_{g(s_a(P))}] $$
  as classes in $K^S$, where we insist too that any equivariant tiles 
  only occur in the rightmost $c$ columns.
\end{Theorem}

\begin{proof}
  The $(a+b)$-sortedness is manifest, and we apply lemma \ref{lem:slicetoperm}
  to obtain $s$, then apply theorem \ref{thm:partialKIP} to get the sum.

  However, that theorem gives the $T^n$-equivariant $K$-class of $\Pi_{\sigma'}$,
  not the coarser $S$-equi\-variant class.
  To get the $K^S$-class, we specialize the $y_i$ for the $b$ and $d$
  groups of coordinates to $t$, and those for the $a$ and $c$ to $0$.
  That kills any summand with an equivariant piece outside the 
  rightmost $c$ columns, and otherwise specializes $\weight_K(P)$ to $\weight_S(P)$.
\end{proof}

We could have applied theorem \ref{thm:partialKIP} all the way to $i=0$, 
rather than filling in tiles only below row $i=a$, but the
specialization of that formula from $K^T$ to $K^S$ would not be
positive in the sense of \cite{AGM}, because of this nonmonotonicity.

These $(a+1,a+b)$-partial pipe dreams from this theorem
will be the lower halves of the DS pipe dreams appearing in our main theorem
\ref{thm:PSPipeDreams}.

\subsection{Duality at mid-sort}\label{ssec:midsort}

Before going to $i=0$ in theorem \ref{thm:partialKIP} we make use
of Grassmannian duality.

\begin{Proposition}\label{prop:midsortduality}
  Consider the $(a+1,a+b)$-partial pipe dreams $\calP$ from theorem
  \ref{thm:bpartial}.  For each $P \in \calP$, not only is
  $\Pi_{g(s_a(P))}$ interval positroid, but its rotation by $a$ is
  \emph{dual} interval positroid, and that dual is $a$-sorted.
\end{Proposition}

\begin{proof}
  First we look at $s_a(P)$, 
  whose labels below the squares $(1,1)\ldots (a,a)$ agree
  with those on $s_{a+b}(\sigma')$, and hence are all $0$s.
  We claim $g(s_a(P))$ has the form
%  
%  \centerline{\epsfig{file=midsort.eps,height=1.5in}}
%  
  
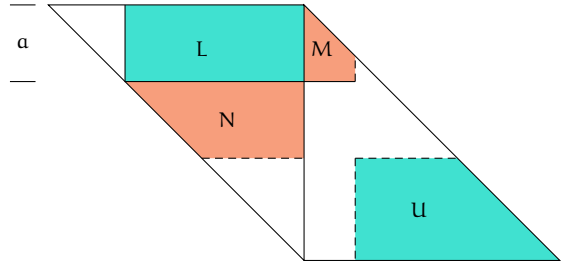
\begin{center}
  \begin{tikzpicture}[scale=.34]
    \fill[Turquoise] (10,10) -- (3,10) -- (3,7) -- (10,7) -- (10,10);
    \fill[Turquoise] (20,0) -- (16,4) -- (12,4) -- (12,0) -- (20,0);
    \fill[Melon] (10,7) -- (3,7) -- (6,4) -- (10,4) -- (10,7);
    \fill[Melon] (10,7) -- (12,7) -- (12,8) -- (10,10) -- (10,7);
    \draw (10,10) -- (0,10) -- (10,0) -- (20,0) -- (10,10) -- (10,0);
    \draw (3,10) -- (3,7) -- (12,7);
    \draw [densely dashed] (12,7) -- (12,8);
    \draw [densely dashed] (12,4) -- (12,0);
    \draw [densely dashed] (12,4) -- (16,4);
    \draw [densely dashed] (6,4) -- (10,4);
    \draw (6,8.3) node[] {\tiny$L$};
    \draw (10.7,8.3) node[] {\tiny$M$};
    \draw (7,5.5) node[] {\tiny$N$};
    \draw (14.5,2) node[] {\tiny$U$};
    \draw (-.5,10) -- (-1.5,10); 
    \draw (-.5,7) -- (-1.5,7); 
    \draw (-1,8.5) node[] {\tiny$a$};
  \end{tikzpicture}
  \end{center}
  \noindent (where $L$ and $M$ have height $a$ and the dashed lines are to indicate that the heights of
  $N$ and $U$ and the widths of $M$ and $U$ are not determined by $a,b,c,d$), with dots NW/SE in each of the
  shaded regions $L,M,N,U$ (and only there). Why?
  \begin{itemize}
  \item $g($an $a$-slice$)$ is $a$-sorted, 
    so $N$ is NW/SE and there are no dots East of $N$.
  \item The $a$-slice only has one kind of letter, $R$, so the dots 
    in $L$ must run NW/SE.
  \item The East half is NW/SE by the definition of $g()$. Hence if
    $M$ is the region above $N$'s rows, and $U$ the region below,
    then $M$ must be West of $U$ and $M,U$ must each run NW/SE.
  \end{itemize}
  Moreover, every row of $N$ and $U$ have dots, since those are the
  only places where the dots could be in those rows.

  Now consider the window on $g(s_a(P))$ rotated forward by $a$:

%  \centerline{\epsfig{file=midsort2.eps,height=2in}}

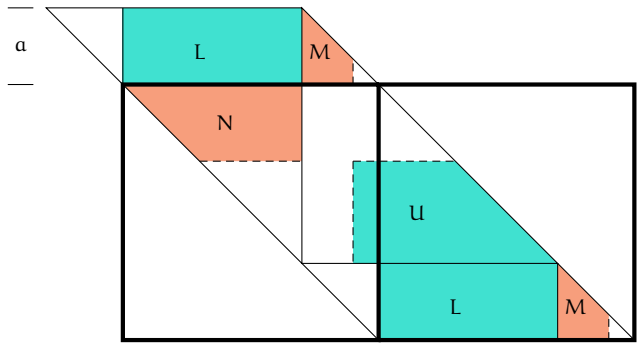
  \begin{center}
  \begin{tikzpicture}[scale=.34]
    \foreach \offset in {(0,0),(10,-10)}
    {
    \begin{scope}[shift={\offset}]
      \fill[Turquoise] (10,10) -- (3,10) -- (3,7) -- (10,7) -- (10,10);
      \fill[Melon] (10,7) -- (12,7) -- (12,8) -- (10,10) -- (10,7);
      \draw (6,8.3) node[] {\tiny$L$};
      \draw (10.7,8.3) node[] {\tiny$M$};
    \end{scope}
    }
    \fill[Turquoise] (20,0) -- (16,4) -- (12,4) -- (12,0) -- (20,0);
    \fill[Melon] (10,7) -- (3,7) -- (6,4) -- (10,4) -- (10,7);
    \draw (10,10) -- (0,10) -- (10,0) -- (20,0) -- (10,10) -- (10,0);
    \draw (20,0) -- (23,-3);
    \draw (10,0) -- (13,-3);
    \draw (3,10) -- (3,7) -- (12,7);
    \draw [densely dashed] (12,7) -- (12,8);
    \draw [densely dashed] (12,4) -- (12,0);
    \draw [densely dashed] (12,4) -- (16,4);
    \draw [densely dashed] (6,4) -- (10,4);
    \draw [densely dashed] (22,-3) -- (22,-2);
    \draw (20,-3) -- (20,0);
    \draw (7,5.5) node[] {\tiny$N$};
    \draw (14.5,2) node[] {\tiny$U$};
    \draw (-.5,10) -- (-1.5,10); 
    \draw (-.5,7) -- (-1.5,7); 
    \draw (-1,8.5) node[] {\tiny$a$};
    \draw [ultra thick]  (13,7) -- (3,7) -- (3,-3) -- (13,-3) -- (13,7) -- (23,7) -- (23,-3) -- (13,-3);
  \end{tikzpicture}
  \end{center}

  The dots in the West square run NW/SE, which says that the {\em dual}
  (obtained by rotating the two-square domino by $180^\circ$,
  as explained in proposition \ref{prop:duality}) is an interval
  positroid variety. Moreover, since $U$ runs NW/SE with a dot in
  every row, when we rotate (to dualize) that East square to $\sigma''$,
%  \centerline{\epsfig{file=midsort3.eps,height=2in}}  
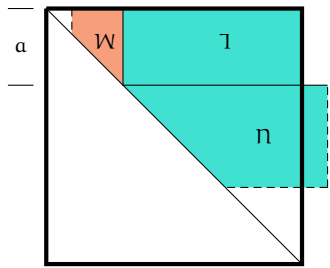
  \begin{center}
  \begin{tikzpicture}[scale=.34]
    \fill[Turquoise] (3,10) -- (3,7) -- (7,3) -- (11,3) -- (11,7) -- (10,7) -- (10,10) -- (3,10);
    \fill[Melon] (1,9) -- (1,10) -- (3,10) -- (3,7) -- (1,9);
    \draw (3,10) -- (3,7) -- (11,7);
    \draw [densely dashed] (11,3) -- (7,3); 
    \draw [densely dashed] (11,3) -- (11,7); 
    \draw [densely dashed] (1,10) -- (1,9);
    \draw (7,8.7) node[rotate=180] {\tiny$L$};
    \draw (2.3,8.7) node[rotate=180] {\tiny$M$};
    \draw (8.5,5) node[rotate=180] {\tiny$U$};
    \draw (0,10) -- (10,0);
    \draw [ultra thick] (0,10) -- (0,0) -- (10,0) -- (10,10) -- (0,10);
    \draw (-.5,10) -- (-1.5,10); 
    \draw (-.5,7) -- (-1.5,7); 
    \draw (-1,8.5) node[] {\tiny$a$};
  \end{tikzpicture}
  \end{center}
  \noindent  we get something $a$-sorted.
\end{proof}

The duality isomorphism $\Gr{a+c}{b+d+c+a} \iso \Gr{b+d}{b+d+c+a}$
is $S$-equivariant if in the latter coordinates, $S$ acts now with 
weight $-1$ on the $b$ and $d$ groups of coordinates (and $0$ on
the others). It remains to compute the classes of
these $\{\Pi_{s''}\}$ in $K^S(\Gr{b+d}{b+d+c+a})$.

\subsection{The final $a$ rows}\label{ssec:finalrows}

\begin{Proposition}\label{prop:finalarows}
  Fix an $(a+1,a+b)$-partial pipe dream $P\in \calP$ from theorem
  \ref{thm:bpartial}, giving the rotated dual $\sigma''$ from the end
  of proposition \ref{prop:midsortduality}.  Split the dots in its
  upper half into $Q$-dots in $M$, $R$-dots in $L$, and let $s$ be
  the corresponding $a$-slice (as in lemma \ref{lem:slicetoperm}).

  Let $\calP'$ be the set of $(0,a)$-partial pipe dreams $P''$
  such that $s_a(P'') = s$, 
  where we use $180^\circ$ rotations of the lower $K$-tiles
  and we insist that any equivariant tiles only occur in the rightmost $b+d$ columns. 

  If we flip the rows of $P'' \in \calP'$ left-right while trading $0$s 
  for $1$s, then the labels below the $a$th row of $P''$ match
  the labels below the $a$th row of the original $P$.
\end{Proposition}

\begin{proof}
  Rather than flipping each time, it will be easier to rotate
  $\sigma''$ to better match the East half of the second figure in
  proposition \ref{prop:midsortduality}. So for short, let $\tau$ be
  this $180^\circ$ rotation of $s_a(\sigma'')$.
  We also rotate the partial permutation $\pi(s_a(\sigma''))$ by $180^\circ$, 
  thus obtaining a partial permutation $\theta$ 
  living in the lower triangle of the ambient square of $\tau$. 

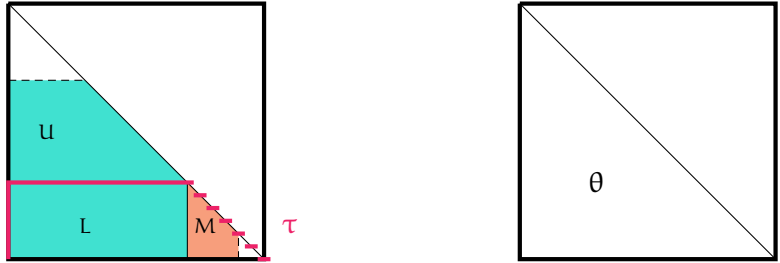
  \begin{center}
  \begin{tikzpicture}[scale=.34]
    \fill[Turquoise] (20,0) -- (13,0) -- (13,-3) -- (20,-3) -- (20,0);
    \fill[Melon] (20,-3) -- (22,-3) -- (22,-2) -- (20,0) -- (20,-3);
    \draw (16,-1.7) node[] {\tiny$L$};
    \draw (20.7,-1.7) node[] {\tiny$M$};
    \draw (24,-1.7) node[WildStrawberry] {\small$\tau$};
    \draw (36,0) node[] {\small$\theta$};
    \fill[Turquoise] (20,0) -- (16,4) -- (13,4) -- (13,0) -- (20,0);
    \draw (13,0) -- (20,0) -- (13,7);
    \draw [densely dashed] (13,4) -- (16,4);
    \draw [densely dashed] (22,-3) -- (22,-2);
    \draw (20,-3) -- (20,0) -- (23,-3);
    \draw (14.5,2) node[] {\tiny$U$};
    \draw [ultra thick]  (13,7) -- (23,7) -- (23,-3) -- (13,-3) -- (13,7);
    \draw [WildStrawberry,ultra thick] (13,-3) -- (13,0) -- (20.25,0);
    \foreach \x in {0,.5,1,1.5,2,2.5}
      \draw [WildStrawberry,ultra thick] (20.25+\x,-.5-\x) -- (20.75+\x,-.5-\x);
    \draw [ultra thick]  (33,7) -- (43,7) -- (43,-3) -- (33,-3) -- (33,7);
    \draw (33,7) -- (43,-3); 
  \end{tikzpicture}
  \end{center}

  We have to inspect the duality operation to see how the labels relate.
  For each column $C \leq b+d+c$ of the rotated slice $\tau$, there is either 
  \begin{itemize}
  \item a dot in the last $a$ rows of $\theta$ (the $L$ part),
    in which case the label in column $C$ of $\tau$ is $R$,
  \item a dot higher up (the $U$ part),
    in which case the label in column $C$ of $\sigma''$ is $0$,
  \item no dot in $C$,
    in which case the label in column $C$ of $\sigma''$ is $1$.
  \end{itemize}
  Correspondingly, in column $a+C$ of $\sigma$, we have
  \begin{itemize}
  \item a dot in row $\leq a$ (the $L$ part), 
    in which case the label in column $a+C$ of $s_a(P)$ is $R$,
  \item no dot in column $a+C$ (since thinking of $U$
    above $N$, they can't both have a dot),
    in which case the label in column $a+C$ of $s_a(P)$ is $1$,
  \item a dot below row $a$ (the $N$ part),
    in which case the label in column $a+C$ of $s_a(P)$ is $0$.
  \end{itemize}

  Now consider the last $a$ columns of $\tau$ (above $M$ in the figure),
  i.e. on the diagonal. These are all $Q$s, then all $1$s.
  The number of $Q$s is the number of dots in $M$, 
  This is also the number of empty rows of $L$, each of
  which comes from a $0$ or $1$ in the left $b+d+c$ columns of $\tau$.
  The labels on the left edge of $\tau$ are $R$s and $Q$s
  the former sitting in the same rows as $L$'s dots, the latter in the same rows as $M$'s. 
  
  Finally, filling the $a$ rows below $\tau$ with lower $K$-tiles
  amounts to filling the $a$ rows above $s_a(\sigma'')$ with $180^\circ$ rotations of them. 
  When doing so, we only place equivariant pieces into the leftmost $b+d$ columns of $\tau$ 
  (i.e., the rightmost $b+d$ columns of $s_a(\sigma'')$)
  since the coordinates in group $a$ have the same weight as those in group $c$, 
  but a different weight from those in groups $b$ and $d$. 
\end{proof}

\subsection{Proof of theorem \ref{thm:PSPipeDreams}}
\label{ssec:DSProof}

This is just a recapitulation of the rest of this section, 
interpreted in terms of DS pipe dreams.

From $\lambda$ and $\mu$, we construct the rotated Richardson variety
in proposition \ref{prop:bigStartingMatrix}, which is $K^S$-homologous
to the direct sum variety of interest. Then we flip and rotate again
to deal with the Richardson variety $\Pi_{\sigma'}$ directly, which to be
equivariant, requires changing the $S$-action. 

Now we can apply theorem \ref{thm:bpartial}, obtaining
$[\Pi_{\sigma'}]$ as a sum over $(a,b)$-partial pipe dreams $P\in \calP$,
which we will take for the lower half of the DS pipe dreams.

Rotate by $a$ and dualize, obtaining bounded juggling patterns
that (by proposition \ref{prop:midsortduality}) are $a$-sorted.
Apply theorem \ref{thm:partialKIP} again with $i=0$,
obtaining each $[\Pi_{\sigma''}]$ as a sum over $(0,a)$-partial pipe dreams.
By proposition \ref{prop:finalarows}), if we flip those 
left-right while exchanging $0 \leftrightarrow 1$,
the bottom of the resulting tiled trapezoid agrees with the top
of the previous tiled trapezoids.

Finally, we need to be sure that the labels on the boundary of the
DS pipe dreams match those specified in theorem \ref{thm:PSPipeDreams};
this is the last part of proposition \ref{prop:finalarows}.

We provide a mnemonic picture explaining the labels 
on the region to be filled by DS pipe dreams. 
The framed region of the juggling pattern of interest gets ripped out, 
its upper half cut off, flipped upside down, and glued back on. 
  
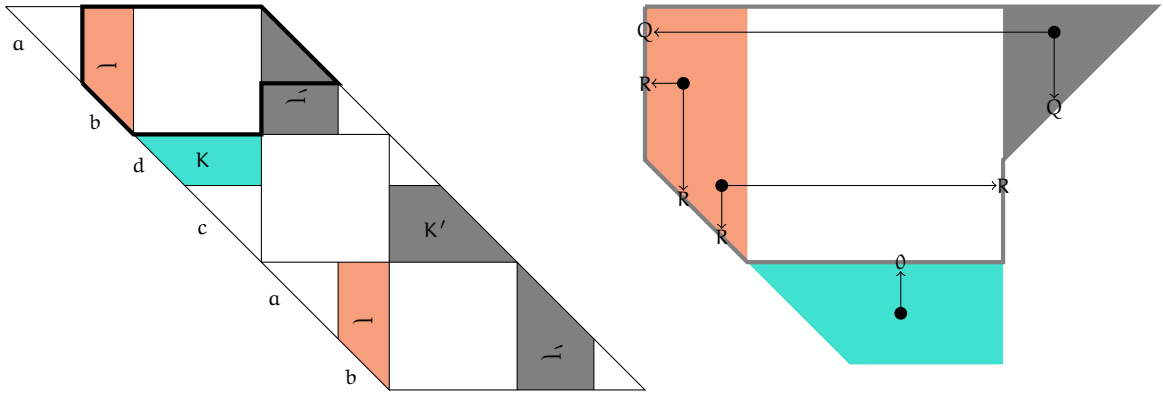
\begin{center}
  \begin{tikzpicture}[scale=.34]
  \foreach \offset in {(25,0)}
  {
    \begin{scope}[shift={\offset}]
    \fill[Turquoise] (0,5) -- (2,3) -- (5,3) -- (5,5) -- (0,5);
    \draw (0,5) -- (5,0) -- (5,5) -- (0,5);
    \draw (2,3) -- (5,3);
    \draw (2.7,4) node[] {\tiny$K$};
    \draw (.2,3.8) node[] {\tiny$d$};
    \draw (2.7,1.3) node[] {\tiny$c$};
    \end{scope}
  }
  \foreach \offset in {(20,5),(30,-5)}
  {
    \begin{scope}[shift={\offset}]
    \fill[Melon] (3,2) -- (3,5) -- (5,5) -- (5,0) -- (3,2);
    \draw (0,5) -- (5,0) -- (5,5) -- (0,5);
    \draw (3,2) -- (3,5);
    \draw (4,2.7) node[rotate=-90] {\reflectbox{\tiny$J$}};
    \draw (.5,3.5) node[] {\tiny$a$};
    \draw (3.5,.5) node[] {\tiny$b$};
    \end{scope}
  }
  \foreach \offset in {(30,5),(40,-5)}
  {
    \begin{scope}[shift={\offset}]
    \fill[Gray] (3,2) -- (3,0) -- (0,0) -- (0,5) -- (3,2);
    \draw (0,5) -- (0,0) -- (5,0) -- (0,5);
    \draw (3,2) -- (3,0);
    \draw (1.4,1.5) node[rotate=-90] {\reflectbox{\tiny$J'$}};
    \end{scope}
  }
  \foreach \offset in {(35,0)}
  {
    \begin{scope}[shift={\offset}]
    \fill[Gray] (2,3) -- (0,3) -- (0,0) -- (5,0) -- (2,3);
    \draw (0,5) -- (0,0) -- (5,0) -- (0,5);
    \draw (2,3) -- (0,3);
    \draw (1.8,1.3) node[] {\tiny$K'$};
    \end{scope}
  }
  \draw (25,10) -- (30,10);
  \draw (35,-5) -- (40,-5);
    \draw[ultra thick] (23,7) -- (25,5) -- (30,5) -- (30,7) -- (33,7) -- (30,10) -- (23,10) -- (23,7);
  \begin{scope}[shift={(45,10)}]
    \fill[Melon] (0,0) -- (0,-6) -- (4,-10) -- (4,0) -- (0,0);
    \fill[Turquoise] (4,-10) -- (8,-14) -- (14,-14) -- (14,-10) -- (4,-10);
    \fill[Gray] (14,0) -- (14,-6) -- (20,0) -- (14,0);
    \draw[Gray,ultra thick] (0,0) -- (0,-6) -- (4,-10) -- (14,-10) -- (14,-6) -- (20,0) -- (0,0);
  \fill (1.5,-3) circle (0.25);
  \draw[->] (1.5,-3) -- (.25,-3);      
  \draw[->] (1.5,-3) -- (1.5,-7.2);
  \fill (3,-7) circle (0.25);
  \draw[->] (3,-7) -- (13.75,-7);
  \draw[->] (3,-7) -- (3,-8.7);
  \fill (10,-12) circle (0.25);
  \draw[->] (10,-12) -- (10,-10.35);
  \fill (16,-1) circle (0.25);
  \draw[->] (16,-1) -- (.35,-1);
  \draw[->] (16,-1) -- (16,-3.6);
  \draw (0,-3) node[] {\tiny$R$};
  \draw (1.5,-7.5) node[] {\tiny$R$};
  \draw (3,-9) node[] {\tiny$R$};
  \draw (14,-7) node[] {\tiny$R$};
  \draw (10,-10) node[] {\tiny$0$};
  \draw (16,-4) node[] {\tiny$Q$};
  \draw (0,-1) node[] {\tiny$Q$};
  \end{scope}
  \end{tikzpicture}
  \end{center}
\begin{itemize}
  \item Each dot in $K$ sees a $0$ on the horizontal edge North of it. 
  All other horizontal labels there are $1$s, hence the word $\mu$. 
  \item Each dot in the lower half of $J$ sees an $R$ on the vertical edge East of it. 
  All other labels there are $0$s, hence the word $\lambda_1$. 
  \item Each dot in the upper half of $J$ sees an $R$ on the vertical edge West of it. 
  Each dot in $J'$ sees a $Q$ on the vertical edge West of it. 
  Every row in the upper half contains a dot, hence the word $\lambda_2$.
  \item Each dot in $J$ sees an $R$ on the horizontal edge South of it. 
  Since every column in $J$ contains a dot, there are only horizontal $R$s South of $J$. 
  \item The labels on vertical edges South of $J$ are all $0$s. 
  \item Each dot in $J'$ sees a $Q$ on the horizontal edge South of it. 
  Since only the first few columns on the West of $J'$ contain dots, 
  this makes for a few $Q$s, followed by $0$s, on horizontal edges South of $J$. 
  \item The labels on vertical edges South of $J'$ are all $1$s. 
\end{itemize}

\end{document}